\documentclass[11pt,a4paper,reqno]{amsart}
\usepackage[foot]{amsaddr}
\usepackage[top=1.5in, left=1.1in, bottom=1in, right=1.2in]{geometry}
\usepackage{amsmath, amsthm, amssymb, amsfonts}
\usepackage{enumitem}
\usepackage{mathtools}
\usepackage{mathrsfs}
\usepackage[usenames, dvipsnames]{color}
\usepackage[T1]{fontenc}
\usepackage[utf8]{inputenc}
\usepackage[square, numbers, comma, sort&compress]{natbib}

\usepackage{epigraph}
 \usepackage{setspace}
\usepackage{subfiles}
\usepackage{graphicx}
\usepackage{hyperref}
\hypersetup{colorlinks=true, urlcolor=blue, linkcolor= blue, citecolor= green, anchorcolor= blue, plainpages=true, hypertexnames=true}

\newcommand{\R}{\mathbb{R}}
\newcommand{\B}{\mathbb{B}}
\newcommand{\N}{\mathbb{N}}
\newcommand{\norm}[1]{\left \| #1  \right \|}
\newcommand{\paren}[1]{\left (  #1 \right ) }
\newcommand{\smap}[3]{#1 : \mathbb{R}^{#2} \longrightarrow \mathbb{R}^{#3} }
\newcommand{\mmap}[3]{#1 : \mathbb{R}^{#2} \rightrightarrows \mathbb{R}^{#3} }
\newcommand{\gph}[1]{ \mathrm{gph} \, #1 }
\newcommand{\rfp}[2]{ (\bar{#1}, \bar{#2}) }

\newcommand{\NUV}{neighborhoods $U$ of $\bar{x} $ and $V $ of $\bar{y}$ }
\newcommand{\mfa}{ \mathrm{~~~~for~all~} }
\newcommand{\at}[2]{at $ \bar{#1} $ for $ \bar{#2} $}

\theoremstyle{plain}
\newtheorem{thm}{\bf Theorem}[section]
\newtheorem{lem}[thm]{\bf  Lemma}
\newtheorem{prop}[thm]{\bf Proposition}
\newtheorem{cor}[thm]{\bf Corollary}

\theoremstyle{definition}

\newtheorem{eg}[thm]{\bf Example}

\newtheorem{rem}[thm]{\bf Remark}

\renewenvironment{proof}{\emph{\textbf{Proof.}} \normalfont}{\hfill $\Box$}

\allowdisplaybreaks

\begin{document}

\title[solution trajectories of generalized equations]{Existence and continuity of solution trajectories of generalized equations with application in electronics}


\author{I. Mehrabinezhad}
\email{i.mehrabinezhad@campus.unimib.it}

\author{R. Pini$^*$}
\address{$^*$ Dipartimento di Matematica e Applicazioni, Via Cozzi 55, 20125 Milano-I}
\email{rita.pini@unimib.it}

\author{A. Uderzo}
\email{amos.uderzo@unimib.it}


\date{\today}



\begin{abstract}
We consider a special form of parametric generalized equations arising from electronic circuits with AC sources and study the effect of perturbing the input signal on solution trajectories. Using methods of variational analysis and strong metric regularity property of an auxiliary map, we are able to prove the regularity properties of the solution trajectories inherited by the input signal. Furthermore, we establish the existence of continuous solution trajectories for the perturbed problem. This can be  achieved via a result of uniform strong metric regularity for the auxiliary map.
\end{abstract}

\maketitle
\setstretch{1.3} 
\section{Introduction}
This paper deals with \emph{parametric generalized equations} of the form
\begin{equation}\label{GE-int}
0 \in f(x) - p(t) + F(x),
\end{equation}
where $\smap{f}{n}{n}$ is a (usually) smooth function, $ \smap{p}{}{n} $ is a function of parameter $t \in [0, 1] $, and $\mmap{F}{n}{n}$ is a set-valued map with closed graph. \\
Generalized equations (when $p$ is a constant function) has been well studied in the literature of variational analysis (see, for instance, \citep{implicit, Mordukhovich} and references therein). Robinson in \citep{robinson1979generalized, robinson1982generalized, robinson1983generalized, robinson} studied in details the case where $F$ is the normal cone at a point $ x \in  \R^n $ (in the sense of convex analysis) to a closed and convex set, and found the setting of generalized equations as an appropriate way to express and analyse problems in complementarity systems, mathematical programming, and variational inequalities.\\
In recent years, the formalism of generalized equations has been used to describe the behaviour of electronic circuits (\citep{adly2, adly, adly3}). In the study of electrical circuits, power supplies (that is, both current and voltage sources) play an important role. Not only their failure in providing the minimum voltage level for other components to work would be a problem, but also small changes in the provided voltage level will affect the whole circuit and the goal it has been designed for. These small changes around a desired value could happen mainly because of failure in precise measurements, ageing process, and thermal effects (see \citep{Iman2017, Sedra}).
Thus, based on the type of voltage/current sources in the circuits, two different cases may be considered:
\begin{enumerate}[topsep=0ex, itemsep=0ex, partopsep=1ex, parsep=1ex, leftmargin = 2ex]
\item[1.] \textbf{static case:}\index{problem ! static case}\\
This is the situation when the signal sources in the circuit are DC (that is, their value is not changing with respect to time).
For practical reasons, we would prefer to rewrite \eqref{GE-int} as $ p \in f(x) + F(x) $, where $ p \in \R^n $ is a fixed vector representing the voltage or current sources in the circuit, and $ x \in \R^n $ represents the mixture of $ n $ variables which are unknown currents of branches or voltages of components. The corresponding \emph{solution mapping} can be defined as follows:
\begin{equation} \label{solution mapping-int}
p \longmapsto S(p) := \left \{ x \in \R^n ~|~ p \in f(x) \, + F(x) \right \}.
\end{equation}
In this framework, small deviations of $ x $ with respect to perturbations of $p$ around a presumed point $ \rfp{p}{x} \in \gph{S} $ could be formulated in terms of local stability properties of $S$ at $\bar{p}$ for $\bar{x}$.
Providing some first order and second order criteria to check the local stability properties of a set-valued map has been the subject of many papers \citep{adly2, adly, durea2012openness, Mordukhovich} to mention a few.
\item[2.] \textbf{dynamic case:} \index{problem ! dynamic case}\\
When an AC signal source (that is, its value is a function of time) is in the circuit, the problem could be more complicated. First of all, the other variables of the model would become a function of time, too. Second, it is not appropriate any more to formulate the solution mapping as $S(p)$. One can consider a parametric generalized equation like \eqref{GE-int} where $ p $ now depends on a scalar parameter $ t \in [0, 1]$\footnote{
In fact, $t$ can belong to any finite interval like $ [0, T]$ for a $T>0$. The starting point $ t = 0$ is considered as the moment that the circuit starts working, in other words, when the circuit is connected to the signal sources and is turned on with a key. We keep the time interval as $ [0, 1] $ in the entire paper for simplicity.\\
},
and define the corresponding solution mapping as
\begin{equation}
S : t \mapsto S(t) = \{ x \in \R^n ~|~  - p(t) + f(x) + F(x) \ni 0 \}.
\end{equation}
The third difficulty rises here: the study of the effects of perturbations of $p $ is not equivalent any more to searching the local stability properties of $S$. \\
Clearly, in this framework for any fixed $t \in [0, 1]$ one has a static case problem. This approach is well known and well studied in the literature, both as a pointwise study (see for instance
\citep{Mordukhovich, Bianchi2013279, Artacho20101149, uderzo2009some} and references therein),
or as a numerical method and for designing algorithms (see for example
\citep{ferreira2016kantorovich, ferreira2016unifying, dontchev2010newton, adly2016newton, adly2015newton}).
In this paper, instead of looking at the sets $S(t)$, we focus on \emph{solution trajectories}, functions like $ x : [0, 1] \to \R^n $ such that
$ x(t) \in S(t),  \mfa t \in [0, 1] $, that is, $x(\cdot) $ is a selection for $S$ over $ [0,1] $.
\end{enumerate}
The main aim of this paper is to investigate the dependence of the solution trajectories on the regularity and changes of the input signal $p$. A side goal is to develop a relationship between pointwise stability properties and their \lq\lq uniform\rq\rq \,version. This has been done with reference to the property of strong metric regularity. \\
The paper is organized as follows: In Section \ref{Preliminaries}, we provide the preliminary definitions. in Section \ref{modeling section}, we give the sketch of the mathematical model for perturbation study of the input signal in the circuits. In Section \ref{Time-varying case}, which contains the main contributions of this paper, we answer questions related to the existence of selections which are regular functions, their relationship with the input signal (Subsection \ref{Continuity of Solution Trajectories}), and their reaction to the small perturbations of the input signal (Subsection \ref{Perturbations of the Input Signal}). We also extend a previous result in \citep{dontchev2013} about uniform strong metric regularity with respect to $t$ over an interval in Subsection \ref{Uniform Strong Metric Regularity-subsection}.

\section{Preliminaries}\label{Preliminaries}
In working with $ \R^n $ we will denote by $ \| x \| $ the Euclidean norm associated with the canonical inner product. The closed ball around $ \bar{x} $  with radius $ r $ is $ \mathbb{B}_r(\bar{x}) = \{ x \in \R^n ~|~~ \| x - \bar{x} \|  \leq r \} $, where the open ball is denoted by int $\mathbb{B}_r(\bar{x})$. We denote the closed unit ball $ \B_1(0) $ by $ \B $.\\
The \textit{graph}, \textit{domain}, and \textit{range} of a given set-valued map $ \mmap{F}{n}{m} $ are defined, respectively, by
gph $ F := \{ (x,y) \in \R^n \times \R^m ~|~ x \in \R^n ~,~ y \in F (x) \}$,
dom $ F := \{ x \in \R^n  ~|~  F (x) \neq \emptyset \}$, and rge $ F := \{ y \in \R^m ~|~ y \in F(x) ~~\mathrm{for~some}~ x \in \R^n \}$. \\
A mapping $ \mmap{S}{m}{n} $ is said to have the \emph{Aubin property}
at $ \bar{y} \in \R^m $ for $ \bar{x} \in \R^n $ if $ \bar{x} \in S(\bar{y}) $, the graph of $ S $ is locally closed at $ (\bar{y}, \bar{x}) $, and there is a constant $ \kappa \geq 0 $ together with neighborhoods $ U $ of $ \bar{x} $ and $ V $ of $ \bar{y} $ such that
\begin{equation}\label{Aubin1}
e(S(y') \cap U, S(y) ) \leq \kappa ~\| y'- y \| \mathrm{~~~for~all~~} y', y \in V,
\end{equation}
where $ e(A, B)  $ is the \emph{excess}\index{excess} of $ A $ beyond $ B $ defined as
\begin{equation}\label{excess}
e \, (A, B) : = \sup_{ x \in A} d(x, B) = \sup_{ x \in A} \, \inf_{y \in B} d(x, y).
\end{equation}
The infimum of $ \kappa $ over all such combinations of $ \kappa, U,$ and $V, $ is called the \emph{Lipschitz modulus} of $ S $ at $ \bar{y}$ for $\bar{x}$ and is denoted by lip $ (S; \bar{y}| \bar{x}) $. \\
$ S $ is said to be \emph{calm} at $ \bar{y} $ for $ \bar{x} $ if $ \rfp{y}{x} \in \gph{S} $, and there is a constant $ \kappa \geq 0 $ along with neighborhoods $ U $ of $ \bar{x} $ and $ V $ of $ \bar{y} $ such that
\begin{equation}\label{calmness1}
e(S(y) \cap U, S(\bar{y})) \leq \kappa ~ \| y - \bar{y} \| \mathrm{~~~~for~all~} y \in V.
\end{equation}
The infimum of $\kappa $ over all such combinations of $\kappa $, $U$ and $V$ is called the \textit{calmness modulus} of $S$ at $\bar{y}$ for $\bar{x}$ and is denoted by clm $(S; \bar{y}|\bar{x})$.\\
A mapping $ S $ is said to have the \textit{isolated calmness} property if it is calm at $ \bar{y} $ for $\bar{x}$ and, in addition, $S$ has a graphical localization at $ \bar{y} $ for $\bar{x}$ that is single-valued at $\bar{y}$ itself (with value $\bar{x}$). Specifically, this refers to the existence of a constant $ \kappa \geq 0 $ and neighborhoods $ U$ of $\bar{x}$ and $V$ of $\bar{y}$ such that
\begin{equation}
\| x - \bar{x} \| ~\leq \kappa~ \| y - \bar{y} \| \mathrm{~~~~when~} x \in S(y) \cap U \mathrm{~~and~~}  y \in V.
\end{equation}
A mapping $ \mmap{F}{n}{m} $ is said to be \emph{metrically regular} at $\bar{x}$ for $\bar{y}$ when $ \bar{y} \in F(\bar{x})$, the graph of $F$ is locally closed at $ \rfp{x}{y} $, and there is a constant $ \kappa \geq 0 $ together with neighborhoods $U$ of $\bar{x}$ and $V$ of $\bar{y} $ such that
\begin{equation}\label{metric regularity}
d \big(x, F^{-1} (y) \big) \leq \kappa~ d(y, F(x)) \mathrm{~~~whenever~~} (x,y) \in U \times V.
\end{equation}
The infimum of $\kappa $ over all such combinations of $\kappa $, $U$ and $V$ is called the \textit{regularity modulus} of $F$ at $\bar{x}$ for $\bar{y}$ and is denoted by reg $(F; \bar{x}|\bar{y})$.\\
A mapping $ F $ with $ \rfp{x}{y} \in \gph{F} $ whose inverse $ F^{-1} $ has a Lipschitz continuous single-valued localization around $\bar{y}$ for $\bar{x}$ will be called \emph{strongly metrically regular} (SMR for short) at $\bar{x}$ for $\bar{y}$.
Indeed, strong metric regularity is just metric regularity plus the existence of a single-valued localization of the inverse (see \citep[Proposition 3G.1, p. 192]{implicit}). \\
$ F $ is called \emph{metrically sub-regular} \at{x}{y} if $ \rfp{x}{y} \in \gph{F} $ and there exists $ \kappa \geq 0 $ along with \NUV such that
\begin{equation} \label{metric sub-regularity}
d ( x, F^{-1} (\bar{y})) ~ \leq ~ \kappa \, d (\bar{y}, F(x) \cap V )  \mfa x \in U .
\end{equation}
The infimum of all $\kappa$ for which (\ref{metric sub-regularity}) holds is the \emph{modulus} of metric sub-regularity, denoted by $ \mathrm{subreg\,} ( F ; \bar{x} | \bar{y}) $. Finally, $F$ is said to be \emph{strongly metrically sub-regular} \at{x}{y} if $\rfp{x}{y} \in \gph{ F} $ and there is a constant $ \kappa \geq 0 $ along
with \NUV such that
\begin{equation} \label{smsr}
\| x - \bar{x} \| \leq \, \kappa \, d(\, \bar{y}, \,F(x) \cap V) ~ \mfa x \in U.
\end{equation}
For a function $ f : \R^d \times \R^n \longrightarrow \R^m $ and a point $ \rfp{p}{x} \in \mathrm{int~dom}\, f$, a function $ \smap{h}{n}{m} $ is said to be an \emph{estimator} of $f$ with respect to $x$ uniformly in $p$ at $ \rfp{p}{x}$ with constant $\mu$ if $  h(\bar{x}) = f (\bar{p},\bar{x}) $ and
$ \widehat{\mathrm{clm}}_x \, (e; (\bar{p}, \bar{x})) \leq \mu < \infty $ for $ e(p,x) = f (p,x) - h(x) $, where $ \widehat{\mathrm{clm}}_x  $ is the uniform partial calmness modulus
\begin{equation*}
\widehat{\mathrm{clm}}_x \, (e; (\bar{p}, \bar{x}))  := \limsup_{\mathclap{\substack{ x \rightarrow \bar{x}, p \rightarrow \bar{p}\\
																							 (p,x)\, \in \, \mathrm{dom }\, e, x \not = \bar{x} }}}
																							   \dfrac{\| e (p,x) - e (p, \bar{x}) \|}{\|x - \bar{x}\|}.
\end{equation*}
It is a \emph{strict estimator} in this sense if the stronger condition holds that
\begin{equation*}
 \widehat{\mathrm{lip}}_x \, (e; (\bar{p}, \bar{x})) \leq \mu < \infty ~~ \mathrm{~~~~~~~~~for~~} e(p,x) = f (p,x) - h(x),
\end{equation*}
and similarly, $ \widehat{\mathrm{lip}}_x $ is the uniform partial Lipschitz modulus. In the case of $ \mu = 0 $, such an estimator is called a \emph{partial first-order approximation}.

\section{Modelling} \label{modeling section}
In this section we will briefly present the practical problem we would face thorough the paper and provide the appropriate mathematical models for it. \\
An electrical circuit is made of some electrical (or electronic) components\footnote{
The word component, in this context, refers to different materials which exhibit a particular electrical behaviour under an electromagnetic force.}
connected together in a special way with wires to serve a specific duty.
For describing a component $Z$, we would look at the current passing through it, referred to as $i_Z$ and the electric potential difference between its terminals, that is voltage over it, referred to as $v_Z$. \\
 The behaviour of the component under different voltages and various currents is usually described with a graph in the $ i - v $ plane, and referred to as the \emph{$ i - v $ characteristic} of the component. These graphs could be presented with functions, like the resistors with the linear relation $v_R = R i_R$ (where $R > 0 $ is a fixed number called Resistance), or set-valued maps, like Diodes. \\
Based on the components in the circuit, and the exactness of the solution required, there is a huge theory and lots of work done around it till now, in electrical engineering literature. We are not exactly interested in this topic, but we will use the setting and rules of circuit theory to formulate our problem.

\subsection{Static case.}
We start with a simple circuit as shown below (Figure \ref{Zener Diode simple circuit}) which involves a \emph{Zener Diode}. Our aim is to find a relation between the given input voltage ($ E $, here) and the variable $ I $, current, in the circuit using the Kirchoff's voltage and current laws.
\begin{figure}[ht]
  \centering
\includegraphics[width=7cm]{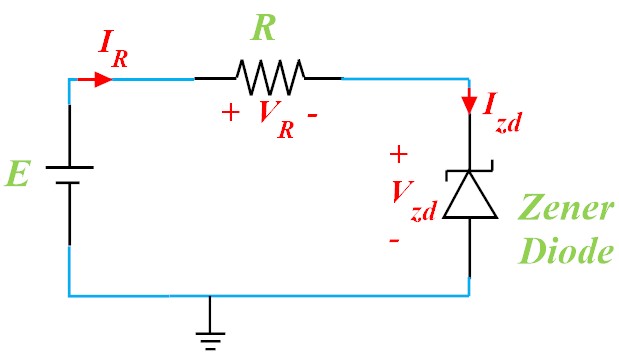}
$ ~~~~~~~~ $
\includegraphics[width=5.5cm]{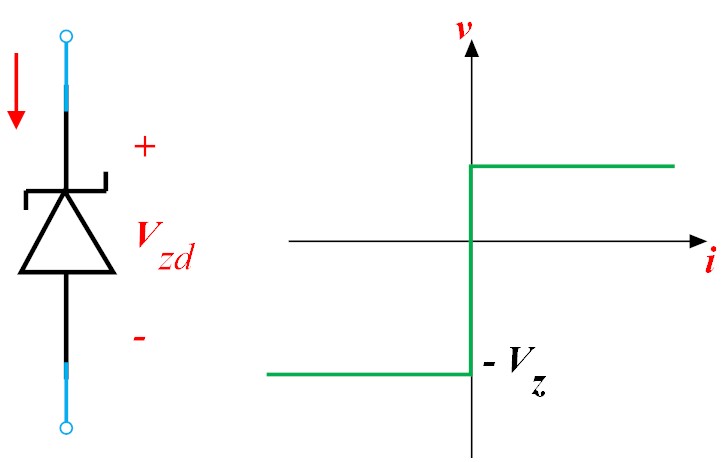}
	\caption{\footnotesize A simple electrical circuit, the schematic and  $i - v$ characteristic of Zener Diode}
	\label{Zener Diode simple circuit}
\end{figure}
\begin{equation}\label{formulation1}
\left. \begin{matrix}
\mathrm{KVL: }  & -E + V_R + V_{zd} = 0~\\
\mathrm{KCL: } & I_R = I_{zd} = I ~~~~~~~~\\
~~& V_R = R \,   I_R \\
~~& V_{zd} \in F(I_{zd})
\end{matrix} \right \}
~~\Rightarrow
~~~E \in R \, I ~ + F(I),
\end{equation}
in which $F(I) $ describes the relation between $I_{zd}$ and $V_{zd}$ as the set-valued map given by the $i - v$ characteristic of the diode. Let us change $E$ to $p$, in order to indicate that it is a \emph{parameter}; and $I$ to $z$, to show that $z$ is the variable, thus we get
\begin{equation}
p \in f(z) + F (z),
\end{equation}
in which $f(z) = R \, z$, in this particular example. For a given $p$, we are interested in the solution mapping defined as
$ S(p) = \left \{  z \in \R ~|~ p \in f(z) + F(z) \right \} $. \\
Regarding different electronic components in the circuit and various circuit schematics, one might need to use some correction matrices to differ those branches that have diodes from others. The general form of the solution mapping in this case would be
\begin{equation}\label{static case}
\begin{matrix}
\Phi (z) : =f(z) \, + B F(Cz) \\
S(p) := \big \{ z \in \R^n ~|~ p \in \Phi (z) \big \},~
\end{matrix}
\end{equation}
where $ p \in \R^n $  is a fixed vector, $ \smap{f}{n}{n} $ is a function, and $\mmap{F}{m}{m}$ is a set-valued map (with certain assumptions), and $B \in \R^{n \times m},~ C \in \R^{m \times n}$ are given matrices. \\
We are going to answer the question concerning the change of source voltage from $\bar{p}$ to $p'$. We try to provide an interpretation of the local stability properties of the solution mapping in terms of the circuit parameters. Let us note that based on the problem one may face during the design process, one of these properties would fit better to his/her demands.

\textbf{Stability Formulation.} Suppose that for a given $\bar{p}$, we know the previous current of \emph{operating point}\footnote{
In the graphical analysis of the circuit, we plot two maps on the same $i - v$ plane: the $i-v $ characteristic of the diode, and the ordered equation (with respect to $i$) of the circuit gained from KVL. The  solution can then be obtained as the equilibrium point, that is the coordinates of the intersection point of the two graphs. This point is called ``operating point''.},
say $\bar{z}$. We also assume that the input change is small, that is, in mathematical terms, $p ' \in \B_r (\bar{p}) =: V$ for some small $r >0$.\\
We are interested in those circuits that keep the small input-change, small. More precisely, the distance between a $z \in S(p')$ and $\bar{z}$, is controlled by the distance $ \|~p' - \bar{p}~\|$. Thus we wish (and search for) $S$ having the following property
\begin{equation} \label{IC-formulation}
 \|~z - \bar{z}~\| \, \leq \, \kappa ~ \|~p' - \bar{p}~\|~~\mathrm{~~when~~~}~z \in S(p') \cap U,~~p' \in V
\end{equation}
where $U$ is a neighborhood of $\bar{z}$, $V$ is a neighborhood of $\bar{p}$, and $\kappa \geq 0 $ is a constant.
This property has already been introduced as isolated calmness. \\
The reason we considered the intersection $ S(p') \cap U $ in the above formulation, is that while it is possible to have different values in $S(p')$, we are just interested in quarantining the existence of a $z$ near enough to $\bar{z}$. \\
$S$ may be not single-valued at $p'$, hence, with a slight modification of the previous formulation, one can obtain the property of calmness as defined in \eqref{calmness1}.
If we are investigating a general local property of the solution mapping and the point $\rfp{p}{z} \in \gph{S} $ does not play a crucial role in our study, we would be interested in the \lq\lq \emph{two-variable}\rq\rq \,version of the previous condition, and hence search for the Aubin property \eqref{Aubin1}.

\textbf{Regularity Formulation.} Up to now, the construction has been built under the assumption that the explicit form of the solution mapping $S(\cdot)$ is in hand and so we can easily calculate values like $S(p')$, which is not true in general. All we are sure we can get from the circuit is $ \Phi = S^{-1} $. It is not always simple (or even possible) to derive the formula of $S$, so we need to provide the proper formulation of the desired properties with $\Phi$.\\
One can start from a point $\rfp{z}{p} \in \gph{\Phi}$, take an arbitrary $z' \in U =: \B_{r_1} (\bar{z}) $, calculate $\Phi(z')$, take a $p \in V =: \B_{r_2} (\bar{p})$, and then ask for the chance of having the output distance $ d \left (  z',\Phi^{-1} (p) \right ) $ being controlled by the input distance $ d \left ( p, \Phi(z') \right ) $, that is,
\begin{equation}
d \left (  z',\Phi^{-1} (p) \right ) \leq \kappa~ d \left ( p, \Phi(z') \right ) \mathrm{~~~whenever~~} (z',p) \in U \times V,
\end{equation}
which is introduced as metric regularity \eqref{metric regularity}. \\Then, make slight modifications by considering the additional condition of single-valuedness at $\bar{z}$ (that is, demanding strong metric regularity); or considering the one-variable version of the above condition (introduced as metric sub-regularity).
The relations that hold between the local stability and metric regularity properties of an arbitrary set-valued map have been well studied in the literature \citep{implicit}. \\
Let us note that if $f$ is not smooth, one can face the problem by means of generalized derivatives as in \citep{izmailov2014, cibulka2016strong}.
\subsection{Dynamic case.}
Let us start with an example. In Figure \ref{fig: A regulator circuit with multiple DC sources}, a simple regulator circuit with a practical model for the diode is shown.
The voltage source is made of $n$ batteries connected to each other in a serial scheme, that provides $ n + 1 $ different levels between $ 0 \, V_s $, and $ 1 \, V_s$.
\begin{figure}[ht]
	\centering
		\includegraphics[width=10.45cm]{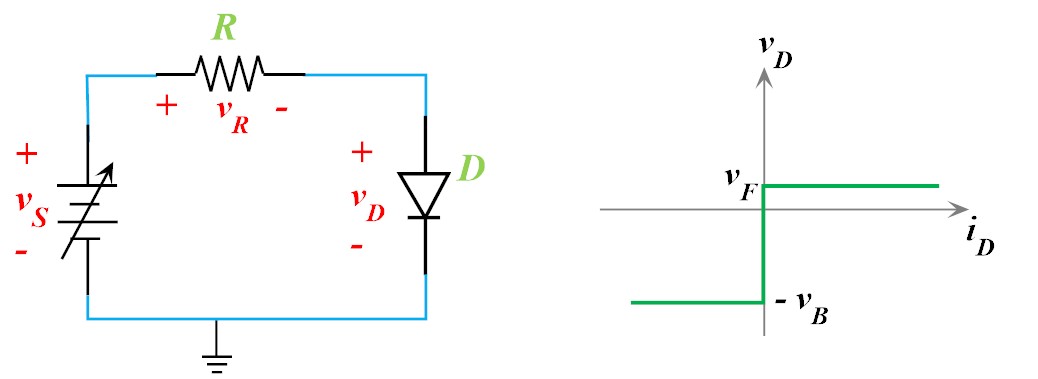}
	\caption{A regulator circuit with multiple DC sources}
	\label{fig: A regulator circuit with multiple DC sources}
\end{figure}
Using Kirchhoff's laws and $ i- v $ characteristics of diode and resistor, we obtain that:
{\small
\begin{equation}
\left.\begin{matrix}
\mathrm{KVL: } & - \frac{m}{n} V_S + V_R+ V_D = 0~ \\
\mathrm{KCL: } & I_{D} = I _R =: I~~~~~~~~~~~\\
~~~~& \, V_R = R \, I_R ~~~~~~~~~~ \\
~~~~& \, V_{D} \in F( I_{D}) ~~~~~~ ~~\\
\end{matrix}\right\}~~\Rightarrow
 \quad 0 \in - p + R z + F (z),
\end{equation}}
where $ p = \frac{m}{n} V_S $, $ z = I $, and $ m = 0, 1, \, \dots , n $, indicates the number of turned-on batteries in the circuit. Then, the solution mapping would be
$$ S(p) = \{ \, z \in \R \, | \, - p + R z + F (z) \ni 0 \, \}. $$
In order to find $S(p)$, we can consider the three parts of $\gph{F}$ separately to solve the generalized equation, fortunately, analytically this time.
\begin{enumerate}
\item[$i.$] For $ z > 0 $, with $ F(z) = \{ v_F \} $. \\
Then, we would have an equation, $ - p + R z + v_F = 0 $. Thus, $ z = \dfrac{p - v_F}{R}$ which is only valid for $ z > 0 $, that is, when $ p > v_F $.

\item[$ii.$] For $ z = 0 $, with $ F(z) = [ - v_B, v_F ] $. \\
Then, $ - p + 0 + [ - v_B, v_F ] \ni 0 $. That is, $ z = 0 $ for $ p \in [ - v_B, v_F ] $.

\item[$iii.$] For $ z < 0 $, with $ F(z) = \{ - v_B \} $. \\
Then, again we would have an equation, $ - p + R z - v_B = 0 $. Thus, $ z = \dfrac{p + v_B}{R}$, as long as $ p < - v_B $.
\end{enumerate}
Therefore, $S$ is a single-valued map in this problem, with the graph shown in Figure \ref{fig: Solution mapping for the circuit} (left), and the rule given as:
{\small
\begin{equation}
S(p) = \left \{
\begin{matrix}
~ \Big \{ \dfrac{p + v_B}{R} \Big \} & & p < - v_B \\
 & & \\
~ \Big \{ 0 \Big \} & & ~~~~~~~p \in [ - v_B, v_F ] \\
 & & \\
~ \Big \{ \dfrac{p - v_F}{R} \Big \} & & p > v_F
\end{matrix} \right.
\end{equation}}
When we deal with an AC voltage source, theoretically we can follow the same procedure. For any $ t \in [0, 1]$, use the specific value $p(t)$ and the transformation graph to find the value of $z$ at that time, that is $z(t)$. Then, we can obtain the graph of $z(\cdot)$ with respect to time, similar to the one shown in Figure \ref{fig: Solution mapping for the circuit} (right) for a sinusoid signal.\\
There are two interesting facts to highlight here:
\begin{enumerate}
\item[\textbf{(a)}] Very naturally, instead of asking for the graph of the solution mapping with respect to the input signal, we focused on the graph of the solution mapping with respect to the time. Of course, when the solution mapping is not a function like this problem, the latter expression needs a clarification.

\item[\textbf{(b)}] Dealing with a function as the input signal, we searched for a function as the output signal. In order to keep the notations consistent, yet without ambiguity, we will refer to these functions as $p(\cdot)$, $\smap{p}{}{n}$, $z(\cdot)$, and so on.
\end{enumerate}

 \begin{figure}[ht]
  \begin{minipage}{.40\textwidth}
  		\includegraphics[width=5.85cm]{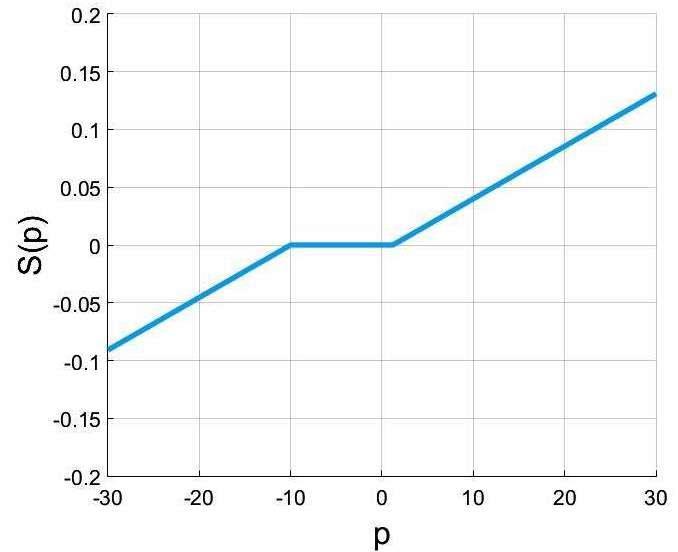}
  \end{minipage}%
  \begin{minipage}{.60\textwidth}
    	\centering
		\includegraphics[width=8.45cm]{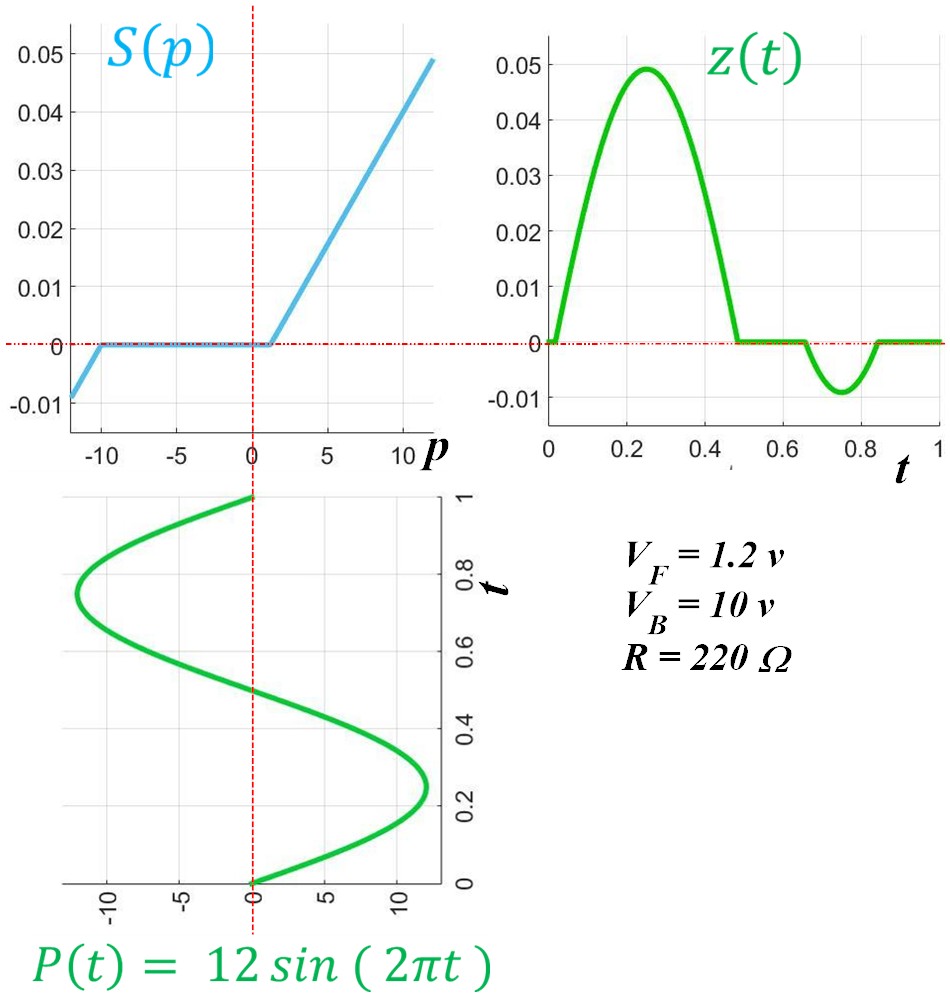}
  \end{minipage}
	\caption{{\small Solution mapping for the circuit in Figure \ref{fig: A regulator circuit with multiple DC sources} (left), graphical method to find the output for a typical input (right)}}
	\label{fig: Solution mapping for the circuit}
\end{figure}

Therefore, in the case of time varying sources, we can assume that $ t $ is a parameter, it belongs to a set like $[0, 1]$, and it would be more appropriate to consider the solution mapping as the (generally set-valued) map that associates to every $ t \in [0, 1]$, the set of all possible vectors $ z$ in $\R^n$ that fits the generalized equation
\begin{equation}\label{pge}
f(z) - p(t) + F(z) \ni 0.
\end{equation}
The \emph{solution mapping} is therefore given by
\begin{equation}\label{solution mapping for time-varying case}
S : t \mapsto S(t) = \{ z \in \R^n ~|~ f(z) - p(t) + F(z) \ni 0 \},
\end{equation}
and a \emph{\textbf{solution trajectory}}\index{solution trajectory} over $[0,1]$ is,
in this case, a function $\bar{z}(\cdot)$ such that $ \bar{z}(t) \in S(t)$  for all $ t \in [0, ~ 1] $, that is, $\bar{z}(\cdot)$ is a selection for $S$ over $[0,1]$. \\
Defining the solution mapping in terms of the parameter $t$, and not directly of the input signal $p$, will cause some difficulties to study the perturbation problem. \\
Comparing to the static case, although for each $t$ one needs to solve a generalized equation of the type discussed in depth in equation \eqref{static case}, the Aubin property of the solution mapping\footnote{
or any other local stability property like calmness or isolated calmness of $S$ or equivalently, metric regularities (all four different definitions) of $S^{-1}$.
}
 at a certain point is not sufficient any more to guarantee the stability of the output with respect to the perturbations of the input signal. In other words, the relation between $S$ and $p(\cdot)$ is not explicitly expressed now.\\
Let us note that when $p$ is a continuous function and under the general assumptions:
\begin{enumerate}
\item[(A1)] $f$ is continuously differentiable in $\R^n$;
\item[(A2)] $F$ has closed graph;
\end{enumerate}
the map $S$ has closed graph. Consider a function $ h: \R \times \R^n \longrightarrow \R^n $, defined as $h(t, v) = - p (t) + f(v) $. Then, the generalized equation (\ref{pge}) can be written as
\begin{equation} \label{parametric generalized equation}
h(t,v) + F(v) \ni 0.
\end{equation}
For any given $(t,z) \in \gph{S}$, define the mapping
\begin{equation} \label{robinson formulation}
v \mapsto G_{t,z}(v) := h(t,z) + \nabla_z h(t,z) \, (v-z) + F(v).
\end{equation}
A point $ (t,z) \in \R^{1+n}$ is said to be a \emph{strongly regular}\index{strongly regular point} point
(see Robinson \citep{robinson}) for the generalized equation (\ref{pge}, or equivalently, \ref{parametric generalized equation}) when $ (t,z) \in \gph {S}$ and the mapping $G_{t,z}$ is strongly metrically regular at $z$ for $0$. That is, there exist constants $a_t, ~ b_t, ~ \lambda_t > 0$ such that the mapping
\begin{equation} \label{smr-constants}
\B_{b_t} (0) \ni y \longmapsto G_{t,z} ^{-1} (y) \cap \B_{a_t} (z(t) )
\end{equation}
is a Lipschitz continuous function with a Lipschitz constant $\lambda_t$.\\
From \cite[Theorem 2B.7, p. 89]{implicit}, one obtains that when $(\bar{t},\bar{z})$ is a strongly regular point for (\ref{pge}), there are open neighborhoods $T$ of $\bar{t}$ and $U$ of $\bar{z}$ such that the mapping
\begin{equation}
T \cap [0, 1] \ni \tau \mapsto S(\tau) \cap U
\end{equation}
is single-valued and Lipschitz continuous on $T \cap [0, ~1]$. Then,
\cite[Theorem 6G.1, p. 426]{implicit} shows that if each point in $ \gph{S} $ is strongly regular, then there are finitely many Lipschitz continuous solution trajectories defined on $ [0, 1] $ whose graphs never intersect each other. In addition, along any such trajectory $ \bar{u}(\cdot) $ the mapping $ G_{t, \bar{u}(t)} $ is strongly regular uniformly in $ t \in [0, 1] $, meaning that the neighborhoods and the constants involved in the definition do not depend on $t$.
Although this theorem is an important result in our study, there is an unpleasant assumption there: $S(t)$  is uniformly bounded. Even if this condition is fulfilled, it is hard to be checked since it requires the whole set $S(t)$ to be clarified and available for any $t \in [0, 1]$.
In the next section it will be shown that the uniform bound could be obtained without this extra assumption.

\section{Time-varying case} \label{Time-varying case}

In this sections we will focus on the problem that is modelled with equations (\ref{pge} - \ref{robinson formulation}). Although we are inspired by Robinson's idea of strongly regular points in defining the auxiliary map (\ref{robinson formulation}), and the techniques in \cite[Theorems 2B.7, 5G.3, and  6G.1]{implicit}, we find it more convenient to do some modifications in the setting in order to adapt it to our problem.\\
Since our aim in this section is the study of the solution trajectories with respect to variations of the input function, $p (t)$, and since working with the function $f(z)$ or with its first order approximation does not play an important role in our case (the proof of this statement will follow soon), we assume to deal with $f(z)$ itself and so to consider the auxiliary mapping
\begin{equation} \label{my auxiliary map}
v \mapsto G_{t}(v) := f(v) - p(t) + F(v).
\end{equation}
For more details on different possible choices of auxiliary maps and how the strong metric regularity would be affected, we state the following proposition:

\begin{prop} [\textbf{Different Auxiliary Maps}] \label{different auxiliary maps} \hfill \\
Suppose that $ \mmap{F}{n}{n}$ is a set-valued map with closed graph, and $ f: \R \times \R^n \longrightarrow \R^n $ is a function admitting $h$ as a strict estimator with respect to $u$ uniformly in $t$, at $\rfp{t}{u}$ with a constant $ \mu $. Given the generalized equation $ f(t,u) + F(u) \ni 0 $, consider the following auxiliary maps:
\begin{eqnarray}
G_{\bar{t},\bar{u}}(v) =h(v) + F(v), \, \\
G_{\bar{t}}(v) = f(\bar{t}, v) + F(v).
\end{eqnarray}
Then, $G_{\bar{t},\bar{u}}$ is SMR at $\bar{u}$ for $0$, if and only if $ G_{\bar{t}}$ is SMR at $\bar{u}$ for $0$, provided that the conditions $ \mu \, \mathrm{reg} \, (G_{\bar{t},\bar{u}} \,;\, \bar{u}) < 1 $, and $ \mu \, \mathrm{reg} \, (G_{\bar{t}} \,;\, \bar{u})  < 1 $ holds.
\end{prop}

\begin{proof}
First observe that, by definition of a strict estimator, $h(\bar{u}) = f(\bar{t},\bar{u})$ and so, $ 0 \in G_{\bar{t},\bar{u}}(\bar{u}) $ is equivalent to $0 \in G_{\bar{t}}(\bar{u}) $. Now taking into account the pointwise relation
$$ G_{\bar{t}}(v) =  f(\bar{t}, v) + F(v) = G_{\bar{t},\bar{u}}(v)  + f(\bar{t}, v) - h(v), $$
one can define a map $ \smap{g_{\bar{t},\bar{u}}}{n}{n} $ with $g_{\bar{t},\bar{u}}(v) = f(\bar{t}, v) - h(v) $.
For any $ v_1, v_2 \in U$ (a neighborhood of $\bar{u}$), we get
\begin{equation*}
\begin{split}
\norm{ g_{\bar{t},\bar{u}}(v_1) - g_{\bar{t},\bar{u}}(v_2) } & = \norm{ f(\bar{t}, v_1) - h(v_1) - f(\bar{t}, v_2) + h(v_2) } \\
& = \norm{ e(\bar{t}, v_1) - e(\bar{t}, v_2) } \\
& \leq \mu \norm{v_1 - v_2 }
\end{split}
\end{equation*}
where the last inequality is obtained by definition of strict estimator. Thus, $g_{\bar{t},\bar{u}}$ is Lipschitz continuous around $\bar{u}$.\\
Now one can use \cite[Theorem 2B.8, p. 89]{implicit} with $G = G_{\bar{t},\bar{u}} $ and $g = g_{\bar{t},\bar{u}} $ and, by assuming that
$  \mu \, \mathrm{reg} \, (G_{\bar{t},\bar{u}} \,;\, \bar{u}) < 1 $, to conclude that $ g + G = G_{\bar{t}} $ has a Lipschitz continuous single-valued localization around $ 0 + g(\bar{u})$ for $\bar{u} $.
Since $g_{\bar{t},\bar{u}} (\bar{u}) = h(\bar{u}) - f(\bar{t},\bar{u}) = 0 $, the latter could be expressed as the SMR of $G_{\bar{t}}$ at $\bar{u}$ for $0$.\\
The converse implication is satisfied in a similar way by letting $G = G_{\bar{t}} $, $g = - g_{\bar{t},\bar{u}} $ and assuming
 $ \mu \, \mathrm{reg} \, (G_{\bar{t}} \,;\, \bar{u})  < 1  $.
\end{proof}

\rem
\textbf{(a)} A closer look at the proof reveals that if $h$ is a strict estimator, then the regularity modulus of $G_{\bar{t}}$ and $ G_{\bar{t},\bar{u}} $ are related to each other with $ \kappa' = \dfrac{\kappa}{1 - \kappa \mu} $.\\
Considering a partially first order approximation of $f$ like $h(v) = f(\bar{t},\bar{u}) + \nabla f_u (\bar{t},\bar{u}) (v - \bar{u}) $, will result in the same modulus for auxiliary maps (since $\mu = 0$ in this case).

\textbf{(b)}
One should note that in general, $h_1(\cdot) = f(\bar{t},\cdot)$ is not a strict estimator of $f$ at the reference point. To guarantee this, one needs an extra assumption like the following:
\begin{center}
$f(t,\cdot)$ is Lipschitz continuous, for any $ t $ in a neighborhood of $\bar{t}$.
\end{center}
However, this is not a necessary condition. For example, in the specific case we are interested in, that is $ h(t,u) = f(u) - p(t) $, $h_1$ is automatically a strict estimator with $\mu = 0$ (in fact, a partial first order approximation).\\

\subsection{Continuity of Solution Trajectories} \label{Continuity of Solution Trajectories}
In this subsection we will discuss the relation between the input signal and solution trajectories under the strong metric regularity assumption of the auxiliary map \eqref{my auxiliary map}.
Throughout the whole subsection we will assume that, given a function $p(\cdot)$, a solution trajectory $z(\cdot)$ exists. \\
Let us start with a simple observation that will be used several times in this chapter. The following easy to prove remark will provide a rule for moving from one auxiliary map to another. This simple yet handy result is a consequence of our choice of auxiliary map and our setting.

\begin{rem} \label{G-relations}
Consider the generalized equation \eqref{pge}, and the auxiliary map (\ref{my auxiliary map}). For arbitrary points $t, t' \in [0, 1] $, the following equalities hold
\begin{eqnarray} \label{set-relations}
G_{t} (v) = G_{t'} (v) + p(t') - p(t), \hspace*{0.2cm}\\
G_{t} ^{-1} (w) = G_{t'} ^{-1} \big( w + p(t) - p(t') \big).
\end{eqnarray}
\end{rem}

In the following proposition, we will prove a continuity result for a given solution trajectory under suitable assumptions. One of the assumptions is that \lq\lq different\rq\rq \,trajectories, that is, trajectories without intersections, may not get arbitrary close to each other. We use an expression based on the graphs of trajectories (see \citep{MRDGE2016}, and check \citep[Example 4.2.7.]{Iman2017} for the difficulties that may arise by \lq\lq bad\rq\rq \,formulations).

\begin{prop} [\textbf{Regularity Dependence of Trajectories on Input Signal}] \label{claim1}
For the generalized equation \eqref{pge}, and the solution mapping \eqref{solution mapping for time-varying case}, assume that
\begin{itemize} [topsep=-1ex,itemsep=0ex,partopsep=1ex,parsep=1ex, leftmargin = 7ex]
\item[(i)] $ z(\cdot) $ is a given solution trajectory which is \textbf{isolated}\index{isolated trajectories} from other trajectories; that is,
there is an open set $ \mathcal{O} \in \R^{n+1} $ such that
\begin{equation} \label{isolated-solution-real}
\{ (t, v) ~|~  t \in [0, 1] \mathrm{~and~} 0 \in G_t (v) \} \cap \mathcal{O} = \gph{ z }.
\end{equation}
\item[(ii)] $ p (\cdot) $ is a continuous function;
\item[(iii)] $G_t$ is pointwise strongly metrically regular; i.e. for any $ t \in [0, 1]$, $G_t$ is strongly metrically regular at $z(t) $ for $0$, with constants $a_t, b_t$, and $\kappa_t > 0 $ defined as \eqref{smr-constants}.
\end{itemize}
Then $ z(\cdot) $ is a continuous function.
\end{prop}

\begin{proof}
Fix $t \in [0, 1]$. We know that $ (t, z(t) ) \in \gph{S} $, so $ 0 \in G_t (z(t)) $ or $ z (t) \in G_t ^{-1} (0)$. For any $\epsilon > 0$, let
 $ \epsilon_1 := \min \, \{ \frac{\epsilon}{\kappa_t}, b_t \}$, in which $b_t$ is the radius of the neighborhood around $0$ in the assumption $(iii)$.
 By the uniform continuity of $p(\cdot)$, there exists $\delta_p > 0 $ such that
\begin{equation*}
\norm{ p(t) - p(\tau) } < \epsilon_1  \mathrm{~~whenever~~} \| \, \tau - t \, \| < \delta_p .
\end{equation*}
Let $ \delta < \delta_p $ and consider $\tau \in [0, 1] $ such that $ \| \, \tau - t \, \| < \delta $. By definition, $z(\tau) \in G_{\tau} ^{-1} (0)$. Using Remark \ref{G-relations} we obtain
$ z(\tau) \in G_t ^{-1} ( p(\tau) - p(t) ) $. \\
On the other hand, by assumption $ (i) $ we also know that $z(\tau) \in \B_{a_t} (z(t)) $. \\
Indeed, assuming $z(\tau) \not \in \B_{a_t} (z(t)) $, allows us to define a Lipschitz continuous function $\widetilde{z}$ as
$$ \widetilde{z} (\tau) := G_t ^{-1} \big( p(\tau) - p(t)  \big) \cap \B_{a_t} (z(t)) $$
on $\B_{\delta} (t) $. By Remark \ref{G-relations}, $ \widetilde{z} (\tau) \in G_{\tau} ^{-1} (0) $ and thus, is (part of) a solution trajectory.
Now, consider a sequence $ (t_n) $ in $\B_{\delta} (t) $ converging to $t$, and recall that, by definition,
$ z(t) = G_t ^{-1} ( 0 ) \cap \B_{a_t} (z(t))$. Thus, $ \widetilde{z} (t_n) \longrightarrow z(t) $.\\
This means that $\widetilde{z}$ is a solution trajectory that could get arbitrarily close to $z(\cdot)$ at $ \big(t, z(t) \big)$, which contradicts assumption $(i)$.\\
So $ z(\tau) \in G_t ^{-1} ( y ) \cap \B_{a_t} (z(t))$ where $ y \in \B_{b_t} (0)$.\\
Now, by assumption $(iii)$, the mapping $  G_t ^{-1} ( \cdot ) \cap \B_{a_t} (z(t)) $ is single-valued and Lipschitz continuous on $  \B_{b_t} (0) $ with Lipschitz constant $\kappa_t$. So
\begin{equation*}
\left \| z(t) - z(\tau) \right \| \, \leq \, \kappa_t \left \| p(t) - p(\tau) \right \| \, < \, \epsilon.
\end{equation*}
Since $t$ was an arbitrary point in $[0, 1]$, the proof is complete.
\end{proof}

\rem
If we assume that $p(\cdot)$ is a Lipschitz continuous function, then following the
previous proof by considering $ \tau_1, \tau_2 \in \B_{\delta} (t) $, we can obtain
\begin{equation*}
\left \| z(\tau_1) - z(\tau_2) \right \| \, \leq \, \kappa_t \norm { \, [p(\tau_1) - p(t) ] - [ p(\tau_2) - p(t) ] \, } \, \leq \, \kappa_t L_p  \, \norm{ \tau_1 - \tau_2 }.
\end{equation*}
This means that $ z(\cdot) $ is locally Lipschitz on $ [0, 1] $ which is a compact set; so it is globally Lipschitz and we can restate the proposition as the following corollary.

\begin{cor} \label{claim1-corollary}
Assume that
\begin{enumerate}
\item [(i)] $z(\cdot) $ is a given continuous solution trajectory;
\item [(ii)] $p (\cdot) $ is a Lipschitz continuous function;
\item [(iii)] $G_t$ is pointwise strongly metrically regular at $z(t) $ for $0$.
\end{enumerate}
Then $ z(\cdot) $ is a Lipschitz continuous function.
\end{cor}

\subsection{Uniform Strong Metric Regularity} \label{Uniform Strong Metric Regularity-subsection}
In this subsection we focus our attention on the uniform strong metric regularity of $ G_t $.
In order to clarify the next statement, we remind that pointwise strong metric regularity of $ G_t $ for all $ t \in [0, 1]$, guarantees for each $t \in [0, 1]$ the existence of constants $a_t, \, b_t, \, \kappa_t > 0$ such that the mapping
\begin{equation*}
\B_{b_t} (0) \ni y \longmapsto G_{t} ^{-1} (y) \cap \B_{a_t} (z(t) )
\end{equation*}
is single valued and Lipschitz continuous with a Lipschitz constant $\kappa_t$. It is worthwhile noting that the radii $a_t, \, b_t $ can be decreased provided that a suitable proportion is kept. The details are expressed in following lemma.

\begin{lem} [\textbf{Proportional Reduction of Radii}] \label{ratio law}
Let $H$ be a strongly metrically regular map at $\bar{x}$ for $\bar{y}$ with a Lipschitz constant $\kappa > 0$ and neighborhoods $ \B_a (\bar{x}) $ and  $ \B_b (\bar{y})$. Then for every positive constants
\begin{center}
 $ a' \leq a $ and $ b' \leq b $ such that $ \kappa b' \leq a' $,
\end{center}the mapping $H$ is strongly metrically regular with the corresponding Lipschitz constant $ \kappa $ and neighborhoods $ \B_{a'} (\bar{x}) $ and  $ \B_{b'} (\bar{y}) $.
\end{lem}

\begin{proof}
Since $ B_{b'} (\bar{y}) \subset B_{b} (\bar{y})$ by assumption, $ H^{-1} (y) \cap \B_a (\bar{x}) = : x $ for every $y \in B_{b'} (\bar{y})$.
Taking into account that $ H^{-1} (\cdot) \cap \B_a (\bar{x}) $ is a Lipschitz continuous function on $ B_{b} (\bar{y}) $, and by definition, $ \bar{x} : = H^{-1} (\bar{y}) \cap \B_a (\bar{x}) $, we get:
\begin{equation*}
\norm{ x - \bar{x} } = \norm{ \big( H^{-1} (y) \cap \B_a (\bar{x}) \big) - \big( H^{-1} (\bar{y}) \cap \B_a (\bar{x}) \big) } \, \leq \, \kappa \norm{ y - \bar{y} } \, \leq \, \kappa b'.
\end{equation*}
So, under the condition $ \kappa b' \leq a' $, we get $ x \in \B_{a'} (\bar{x}) $. \\
Indeed, in this case any $ y \in B_{b'} (\bar{y}) $ will be in the domain of $ H ^{-1} (\cdot) \cap \B_{a'} (\bar{x}) $. Then, the single-valuedness and Lipschitz continuity are the consequences of dealing with the same map (that is, $\gph{H}$).
\end{proof}

\begin{thm} [\textbf{Uniform Strong Metric Regularity}] \label{claim2} \hfill \\
For the generalized equation \eqref{pge}, and the solution mapping \eqref{solution mapping for time-varying case}, assume that
\begin{enumerate}
\item [(i)] $z(\cdot) $ is a given continuous solution trajectory;
\item [(ii)] $p (\cdot) $ is a continuous function;
\item [(iii)] $G_t$ is pointwise strongly metrically regular at $z(t) $ for $0$.
\end{enumerate}
Then there exist constants $a, \, b, \, \kappa > 0$ such that for any $ t \in [0, \,1] $, the mapping
\begin{equation*}
\B_{b} (0) \ni y \longmapsto G_{t} ^{-1} (y) \cap \B_{a} (z(t) )
\end{equation*}
is single valued and Lipschitz continuous with a Lipschitz constant $\kappa$.
\end{thm}

\begin{proof}
We prove the statement in two steps. First, by showing the mentioned map must be single-valued without caring about the Lipschitz regularity, and then by proving it is a Lipschitz continuous function.
\par
\textit{\textbf{STEP 1. Single-valuedness: }}\\
We show that there exist $ a, b > 0 $ such that for any $ t \in [0, 1]$, the map
\begin{equation} \label{usmr-01}
\B_b (0) \ni y \longmapsto G_{t} ^{-1} (y) \cap \B_{a} (z(t) )
\end{equation}
is single-valued.
We argue by contradiction, by assuming that for any $ a, \, b > 0 $, there exists $t_{a,b} \in [0,\, 1] $ such that \eqref{usmr-01} does not hold. In particular, take
$ a_n =\frac{1}{n},  b_n = \frac{1}{n^3} $; then, for every $n \in \N$, there exists $t_n (:= t_{a_n, b_n} ) \in [0, \, 1] $ such that
\begin{equation}
\B_{b_n} (0) \ni y \longmapsto G_{t_n} ^{-1} (y) \cap \B_{a_n} (z(t_n) )
\end{equation}
is not single-valued, which is equivalent to
\begin{enumerate}
\item[\textit{Case 1.}] there exists $ y_n \in \B_{b_n} (0) $ such that the cardinality of the set $  G_{t_n} ^{-1} (y_n) \cap \B_{a_n} (z(t_n) ) $ is grater than one, or
\item[\textit{Case 2.}] there exists $ y'_n \in \B_{b_n} (0) $ such that the set $  G_{t_n} ^{-1} (y'_n) \cap \B_{a_n} (z(t_n) ) $ is empty\footnote{
In other words, the mapping
\begin{equation*}
y \longmapsto G_{t_n } ^{-1} (y) \cap \B_{a_n} (z(t_n ) )
\end{equation*}
for at least a point $ y \in \B_{b_n} (0) $, is empty, or multivalued, that is, it has at least two values.
}.
\end{enumerate}
By replacing $ (t_n) $ with a subsequence (if necessary), from the compactness of $[0, 1] $ in $\R$, we can assume that $ t_n \longrightarrow t_0 \in [0, \, 1]$. We will try to reach a contradiction in each case.

\textit{Case 1. Multi-valuedness}\\
Suppose that, for any $ n \in \N$, there exist $t_n \in [0, \, 1] $ and at least a $y_n \in \B_{b_n } (0)$ such that
 $z_n ^1, z_n ^2 \in G_{t_n } ^{-1} (y_n) \cap \B_{a_n} (z(t_n ) ) $ with $z_n ^1 \not = z_n ^2$. \\
By assumption $(iii)$, there exist constants $a_{t_0}, b_{t_0}, \kappa_{t_0} > 0 $ such that the mapping
\begin{equation*}
\B_{b_{t_0}} (0) \ni w \longmapsto G_{t_0} ^{-1} (w) \cap \B_{a_{t_0}} (z(t_0) )
\end{equation*}
is single valued and Lipschitz continuous with Lipschitz constant $ \kappa_{t_0} $.\\
Make $b_{t_0} > 0$ smaller if necessary so that
\begin{equation}\label{relation1}
 \kappa_{t_0} b_{t_0} \leq a_{t_0}.
\end{equation}
For $n$ large enough (i.e. $ n > N_0 \in \N$), one can have the following:
\begin{equation}\label{relation2}
b_n < \dfrac{b_{t_0}}{2},~ ~ \norm{ p(t_n) - p(t_0) } < \dfrac{b_{t_0}}{2},~ ~ \norm{ z(t_n) - z(t_0) } < \dfrac{a_{t_0}}{2}, ~~ a_n < \dfrac{a_{t_0}}{2}, ~~ \kappa_{t_0} < n^2,
\end{equation}
in which the second and third inequalities are the results of continuity assumptions of $p(\cdot)$ and $z(\cdot)$, respectively. Then,
\begin{equation*}
\norm{ z_n ^1 - z(t_0) } \leq \norm{z_n ^1 - z(t_n)} + \norm{z(t_n) - z(t_0)} \leq  a_n + \frac{a_{t_0}}{2} < a_{t_0}.
\end{equation*}
The same holds for $z_n ^2 $; thus, $z_n ^1, z_n ^2  \in \B_{a_{t_0}} (z(t_0 ) ) $. On the other hand, $ z_n ^i \in G_{t_n } ^{-1} (y_n)$ for $ i = 1, 2$, and Lemma \ref{G-relations} implies that $ z_n ^i \in G_{t_0 } ^{-1} \big(y_n + p(t_n) - p(t_0) \big) $.
But
$$ \norm{ y_n + p(t_n) - p(t_0) } \, \leq \, \norm{y_n - 0} + \norm{p(t_n) - p(t_0)} \, \leq \, b_n + \dfrac{ b_{t_0} } {2} \, < \,  b_{t_0}. $$
Thus,
 $ (y_n + p(t_n) - p(t_0)) \in \B_{b_{t_0}}(0)$, which is a contradiction since, in that neighborhood,
$ G_{t_0 } ^{-1} (\cdot) \cap  \B_{a_{t_0}}(z(t_0)) $ is single-valued. \\
\textit{Case 2. Emptiness}\\
Let us now suppose that, for any $n \in \N$, there exist $t_n \in [0, \, 1] $ and at least a point $y'_n \in \B_{b_n } (0)$ such that $ G_{t_n } ^{-1} (y'_n) \cap \B_{a_n}  (z(t_n ) )$ is empty. \\
For $n$ large enough, the inequalities in (\ref{relation1}) and (\ref{relation2}) hold, and we have already proved that $ y \in \B_{b_n} (0) $ implies
$ y + p(t_n) - p(t_0)) \in \B_{b_{t_0}} (0) $. Therefore, since $y'_n \in \B_{b_n} (0)$, the mapping $  G_{t_0} ^{-1} (y'_n + p(t_n) - p(t_0)) \cap \B_{a_{t_0}} (z(t_0) ) $ is single-valued. In particular, it implies that $G_{t_0} ^{-1} (y'_n + p(t_n) - p(t_0)) \not = \emptyset $.
Let $z$ be a point in $ G_{t_0} ^{-1} (y'_n + p(t_n) - p(t_0))$. Then, by using Lemma \ref{G-relations}, we obtain
$ y'_n \in G_{t_n} (z) $, in particular, $G_{t_n} ^{-1} (y'_n) $ is not empty. The contradiction assumption implies that
\begin{equation}\label{contradiction point}
\norm{ z - z(t_n) } > a_n.
\end{equation}
We will show the inconsistency between the contradiction assumption and the assumptions of the theorem with this inequality. In order to proceed, let us first prove that the mapping
\begin{equation} \label{usmr-02}
\B_{\frac{b_{t_0}}{2}} (0) \ni y \longmapsto G_{t_n} ^{-1} (y) \cap \B_{\frac{a_{t_0}}{2}} (z(t_n) )
\end{equation}
is single-valued and Lipschitz continuous with Lipschitz constant $  \kappa_{t_0} $. \\
As a matter of fact, we have already seen that $ G_{t_n} ^{-1} (y) \neq \emptyset $ for every $ y \in \B_{b_n} (0) $, and $ G_{t_n} ^{-1} (y) \cap \B_{a_{t_0} / 2} (z(t_n) ) $ is not multi-valued. Thus, it only remains to show that
$ G_{t_n} ^{-1} (y) \cap \B_{a_{t_0} / 2} (z(t_n) ) \not = \emptyset $ for every $ y \in \B_{b_{t_0} / 2} (0) $.\\
Denote by $z_y$ the point $ z_y := G_{t_0} ^{-1} \big( y + p(t_n) - p(t_0) \big) \cap \B_{a_{t_0}} (z(t_0) ) $. First observe that, by Lemma \ref{G-relations}, $ z_y \in G_{t_n} ^{-1} (y) $.\\
On the other hand, by definition, $ z(t_n) \in G_{t_n} ^{-1} (0) $ and by using Lemma \ref{G-relations}, we get $ z(t_n) \in G_{t_0} ^{-1} (p(t_n) - p(t_0)) $. We also know that $ z ( t_n) \in \B_{a_{t_0}} (z(t_0) ) $ (from the inequalities in \eqref{relation2}). The single-valuedness of $ G_{t_0} ^{-1} (.) \cap \B_{a_{t_0}} (z(t_0) ) $ over $ \B_{b_{t_0}} (0) $ allows us to write
$ z(t_n) = G_{t_0} ^{-1} \big( p(t_n) - p(t_0) \big) \cap \B_{a_{t_0}} (z(t_0) ) $ without ambiguity. Thus, we have:
\begin{equation*}
\begin{split}
\norm{ z_y - z(t_n) } & = \scalebox{0.9}{ $ \norm{ [ \, G_{t_0} ^{-1}\paren{ y + p(t_n) - p(t_0) } \cap \B_{a_{t_0}} (z(t_0) ) \, ] -  [ \, G_{t_0} ^{-1} (p(t_n) - p(t_0)) \cap \B_{a_{t_0}} (z(t_0) ) \, ] } $ } \\
& \leq \kappa_{t_0} \norm{  y + p(t_n) - p(t_0) - ( p(t_n) - p(t_0) ) } \\
& \leq \kappa_{t_0} \norm{y - 0 } \\
& \leq \, \kappa_{t_0} \frac{b_{t_0}}{2} \, \leq \, \frac{1}{2} a_{t_0} .
\end{split}
\end{equation*}
Which means $ z_y \in G_{t_n} ^{-1} (y'_n) \cap  \B_{\frac{a_{t_0}}{2}} (z(t_n) )$.\\
A similar reasoning reveals the Lipschitz continuity of the map $ G_{t_n} ^{-1} (\cdot) \cap \B_{\frac{a_{t_0}}{2}} (z(t_n) ) $.\\
Indeed, taking any $y_i \in \B_{b_{t_0}/2} (0) $, one can define $z_i:= G_{t_n} ^{-1} (y_i) \cap \B_{a_{t_0}/2} (z(t_n) ) $ for $i =1, 2$ without ambiguity. Using the second and third inequalities in (\ref{relation2}), we have
\begin{equation*}
\begin{split}
& z_i \in G_{t_n} ^{-1} (y_i) = G_{t_0} ^{-1}\paren{ y_i + p(t_n) - p(t_0) } \mathrm{~~~~and~~~~}  y_i + p(t_n) - p(t_0)  \in \B_{b_{t_0}} (0) \\
& z_i \in  \B_{a_{t_0}/2} (z(t_n) ), \mathrm{~~~and~~~} \norm{ z(t_n) - z(t_0) } < a_{t_0}/2, \mathrm{~~so~~} z_i \in  \B_{a_{t_0}} (z(t_0) ).
\end{split}
\end{equation*}
Thus we are allowed to use the single-valuedness and Lipschitz property of $G_{t_0} ^{-1}$ to obtain
\begin{equation*}
\begin{split}
\norm{ z_1 - z_2 } & = \norm{ [ \, G_{t_n} ^{-1} (y_1) \cap \B_{a_{t_0}/2} (z(t_n) ) \, ] -  [ \, G_{t_n} ^{-1} (y_2) \cap \B_{a_{t_0}/2} (z(t_n) ) \, ] } \\
& = \scalebox{0.88}{ $\norm{ [ \, G_{t_0} ^{-1}\paren{ y_1 + p(t_n) - p(t_0) } \cap \B_{a_{t_0}} (z(t_0) ) \, ] -  [ \, G_{t_0} ^{-1} (y_2 + p(t_n) - p(t_0)) \cap \B_{a_{t_0}} (z(t_0) ) \, ] } $ } \\
& \leq \kappa_{t_0} \norm{y_1 - y_2 },
\end{split}
\end{equation*}
which reveals the Lipschitz property of the map in \eqref{usmr-02}.

Now, having the strong metric regularity of $ G_{t_n}(\cdot)$ with constants $ \frac{a_{t_0}}{2}, \frac{b_{t_0}}{2}, \kappa_{t_0} $ in hand, we use Lemma \ref{ratio law} with $ a' = a_n = \frac{1}{n} \leq \frac{a_{t_0}}{2} $, $ b' = b_n = \frac{1}{n^3} \leq \frac{b_{t_0}}{2} $, to obtain the strong metric regularity of
$ G_{t_n}(\cdot)$ with constants $ a_n, b_n, \kappa_{t_0} $ (reminding that the last inequality of (\ref{relation2}) guarantees $  \kappa_{t_0} b' \leq a' $).
Now for the specific $y'_n \in \B_{b_n } (0) $ claimed before, there exists $ z \in G_{t_n} ^{-1} (y'_n) \cap \B_{a_n} (z(t_n) )$ which
contradicts (\ref{contradiction point}).\\
Therefore, till now we have proved that there exist $ a^*,b^*  > 0 $ such that the mapping
\begin{equation*}
\B_{b^*} (0) \ni y \longmapsto G_{t } ^{-1} (y) \cap \B_{a^*} (z(t) )
\end{equation*}
is single-valued for any $ t \in [0, 1] $. \\
\textit{\textbf{STEP 2. Lipschitz Continuity:}}\\
Being sure that we deal with a function, now we proceed by claiming that there exist $ b \leq b^* $, and  $ \kappa > 0 $ such that the mapping
\begin{equation*}
\B_{b} (0) \ni y \longmapsto G_{t } ^{-1} (y) \cap \B_{a^*} (z(t) )
\end{equation*}
is Lipschitz continuous with Lipschitz constant $\kappa $ for all $ t \in [0, 1] $.\\

We will prove the assertion by contradiction. Suppose the claim is false; then, for any $ b \leq b^* $, and any $ \kappa > 0 $, there exists $t_{b, \kappa}  \in [0, \, 1] $ such that the mapping
\begin{equation*}
\B_{b} (0) \ni y \longmapsto G_{t_{b, \kappa} } ^{-1} (y) \cap \B_{a^*} (z(t_{b, \kappa}) )
\end{equation*}
is not Lipschitz with constant $\kappa$. Since we already know that this map is single-valued, the contradiction assumption yields that for every $ \kappa > 0 $, there exist $y_1, y_2 \in \B_b (0) $, with $ y_1 \not = y_2 $ such that
\begin{equation*}
\norm{ \, \left [  G_{t_{b, \kappa} } ^{-1} (y_1) \cap \B_{a^*} (z(t_{b, \kappa}) ) \right ] - \left [  G_{t_{b, \kappa} } ^{-1} (y_2) \cap \B_{a^*} (z(t_{b, \kappa}) ) \right ] \, } \, > \, \kappa \norm{ y_1 - y_2 }.
\end{equation*}
For any $ n \in \N $, let $ b_n : = \min \{\, \frac{1}{n^3}, b^* \, \}, \kappa_n := n $ and set $ t_n := t_{b_n, \kappa _n } $. Then there exist at least
$y_n ^1, y_n ^2 \in \B_{b_n} (0) $, with $ y_n ^1 \not = y_n ^2 $ such that
\begin{equation*}
\left \| y_n ^1 -  y_n ^2 \right \| \, n \, < \, \norm{ \, \left [  G_{t_n } ^{-1} (y_n ^1) \cap \B_{a^*} (z(t_n) ) \right ] - \left [  G_{t_n } ^{-1} (y_n ^2) \cap \B_{a^*} (z(t_n) ) \right ] \, }
\end{equation*}
Let $g_n ^i := G_{t_n } ^{-1} (y_n ^i) \cap \B_{a^*} (z(t_n) ) $ for $ i = 1, 2 $, and assume that $t_n$ converges to a point, say $t_0$. \\
For $n$ large enough, one has the following:
\begin{equation*}\label{relation4}
b_n < \dfrac{b_{t_0}}{2},~ ~ \norm{ p(t_n) - p(t_0) } < \dfrac{b_{t_0}}{2},~ ~ \norm{ z(t_n) - z(t_0) } < \dfrac{a^*}{2},~~ \kappa_{t_0} < n^2.
\end{equation*}
On the one hand, $g_n ^i \in G_{t_n } ^{-1} (y_n ^i) = G_{t_0 } ^{-1}\paren{ \, y_n ^i + p(t_n) - p(t_0) \, } $ and the above conditions imply that
 $ \paren{ \, y_n ^i + p(t_n) - p(t_0) \, } \in \B_{b_{t_0}} (0) $.\\
On the other hand, $ g_n ^i \in B_{a^*} (z(t_n) ) $. We will show that $ g_n ^i \in B_{a_{t_0}} (z(t_0) ) $. \\
Indeed, let $w_n ^i := G_{t_0 } ^{-1}  \paren{\, y_n ^i + p(t_n) - p(t_0)} \cap B_{a_{t_0}} (z(t_0) ) $ for $ i = 1, 2 $. \\
Since $  \paren{\, y_n ^i + p(t_n) - p(t_0)} \longrightarrow 0 $, by the continuity of
$G_{t_0 } ^{-1}  (\cdot) \cap B_{a_{t_0}} (z(t_0) ) $ around $0$, we get
$$ w_n ^i \longrightarrow z(t_0) =  G_{t_0 } ^{-1}  (0) \cap B_{a_{t_0}} (z(t_0) ). $$
Thus, for any $\epsilon > 0 $, there exists $N_{\epsilon} \in \N$ such that, for $ n > N_{\epsilon}$, one has
$ \norm{ w_n ^i - z(t_0) } < \epsilon $. Let $ \epsilon = a^* / 2 $. Then,
\begin{equation*}
\begin{split}
\norm{ w_n ^i - z(t_n) } & \leq \norm{ w_n ^i - z(t_0) } + \norm{ z(t_0) - z(t_n) } \\
& <  a^* / 2 +  a^* / 2 \\
& <  a^*,
\end{split}
\end{equation*}
which means that $ w_n ^i \in \B_{a^*} (z(t_n)) $. Combining with $ w_n ^i \in  G_{t_n } ^{-1} (y_n ^i) $ (obtained by using Lemma \ref{G-relations}),
we get that
$  w_n ^i \in  G_{t_n } ^{-1} (y_n ^i) \cap \B_{a^*} (z(t_n)) $.
Hence, by the single-valuedness of $ G_{t_n } ^{-1} (\cdot) \cap \B_{a^*} (z(t_n)) $, we can conclude that $  w_n ^i =  g_n ^i \in B_{a_{t_0}} (z(t_0) ) $.
Then, the assumption (iii) of the theorem results in $ \norm {g_n ^1 -  g_n ^2} \, \leq \, \kappa_{t_0} \norm{y_n ^1 -  y_n ^2} $. So
\begin{equation*}
\norm{y_n ^1 -  y_n ^2} \, n \, < \, \norm {g_n ^1 -  g_n ^2} \, \leq \, \kappa_{t_0} \norm{y_n ^1 -  y_n ^2},
\end{equation*}
which is a contradiction, since it implies boundedness of the sequence $ (\kappa_n ) := ( n )$. Combining the two steps ends the proof.
\end{proof}

Let us note that under the stronger assumption of Lipschitz continuity for $p(\cdot)$ and $z(\cdot)$, a simpler and more direct proof can be provided. The proof will be in the direction of \cite[Theorem  6G.1]{implicit} without the uniformly boundedness assumption over the sets $S(t)$, and the special structure of the single-valued part here allows us to bypass the use of \cite[Theorem 5G.3]{implicit}; refer to \citep[Theorem 4.2.12]{Iman2017} for more details.

\subsection{Perturbations of the Input Signal} \label{Perturbations of the Input Signal}
In this subsection we try to take into account the small variations of the function $ p (\cdot) $. More precisely, we consider a continuous function $\widetilde{p} (\cdot) $ such that $ \norm { \widetilde{p} (t) - p(t) } < \epsilon $ for any $ t \in [0, \, 1] $, and for a suitably small $ \epsilon > 0 $. We deal with the perturbed form of problem (\ref{pge}). To be more specific, we consider the generalized equation
\begin{equation} \label{perturbed GE}
f(z) - \widetilde{p} (t) + F(z) \ni 0,
\end{equation}
denote the corresponding solution mapping with $\widetilde{S}$,
\begin{equation} \label{perturbed S}
\widetilde{S} : t \mapsto \widetilde{S}(t) = \{ z \in \R^n ~|~ f(z) - \widetilde{p} (t) + F(z) \ni 0 \},
\end{equation}
and define the auxiliary map $ \mmap{ \widetilde{G_t} }{n}{n}$ as $  \widetilde{G_t} (v) = f(v) + F(v) - \widetilde{p}(t) $. \\
The easy-to-check equalities
\begin{eqnarray}
\widetilde{G_t}(v)  = G_t(v)  + p(t) - \widetilde{p}(t) \hspace*{0.4cm} \label{perturbation-relation1} \\
\widetilde{G_t} ^{-1} (w) = G_t ^{-1}  \big( w + \widetilde{p}(t) - p(t) \big)   \label{perturbation-relation2}
\end{eqnarray}
for each $ t \in [0, 1] $, will be useful for connecting the strong metric regularity properties of $\widetilde{G_t} $ to those of $G_t$ as described in the following lemma. Once more, we want to indicate that the straightforward equalities \eqref{perturbation-relation1} and \eqref{perturbation-relation2} are a consequence of our choice of the auxiliary maps and the special form of the single-valued part of the generalized equation \eqref{perturbed GE}.

\begin{lem} [\textbf{Perturbation Effect on the Auxiliary Map}] \label{Perturbation Effect on the Auxiliary Map} \hfill \\
Assume that $ p(\cdot) $ and $ \widetilde{p}(\cdot) $ are continuous functions from $ [0, 1]$ to $ \R^n$ with $ \norm { \widetilde{p} (t) - p(t) } < \epsilon $ for any
$ t \in [0, 1] $. If $G_t$ is strongly metrically regular at $\bar{u}$ for $0$
\big(i.e. $ (\bar{u},0) \in \gph \, G_t $ and there exist constants $a_t, \, b_t, \, \kappa_t > 0$ such that the mapping
\begin{equation*}
\B_{b_t} (0) \ni y \longmapsto G_{t} ^{-1} (y) \cap \B_{a_t} (\bar{u} )
\end{equation*}
is single valued and Lipschitz continuous with Lipschitz constant $ \kappa_t $\big),
then for any positive $ \epsilon < b_t $ the mapping
\begin{equation} \label{perturbed G}
w \longmapsto \widetilde{G_t} ^{-1} (w) \cap \B_{a_t} (\bar{u})
\end{equation}
is a Lipschitz continuous function on
$\B_{b_t - \epsilon} (0)$ with Lipschitz constant $\kappa_t$.
\end{lem}

\begin{proof}
Considering that the intersecting ball $ \B_{a_t} (\bar{u} )$ is the same for both maps $ G_{t} ^{-1} $ and $ \widetilde{G_t} ^{-1} $, the proof should include the following steps:
\begin{enumerate}
\item[1.] the sets $  \widetilde{G_t} ^{-1} (w)  $ and also $ \widetilde{G_t} ^{-1} (w) \cap \B_{a_t} (\bar{u}) $ are not empty for any $ w \in \B_{b_t - \epsilon} (0) $;
\item[2.] the mapping \eqref{perturbed G} is not multivalued;
\item[3.] the mapping \eqref{perturbed G} is a Lipschitz continuous function (with constant $\kappa_t$).
\end{enumerate}

Choose any $w_1, w_2 \in \B_{b_t - \epsilon} (0)$. From assumption we get $ \big( w_i + \widetilde{p}(t) - p(t) \big) \in \B_{b_t} (0) $ for $ i = 1, 2 $. Then, the pointwise strong metric regularity of $G_t$, lets us define
$ u_i := G_t ^{-1}  ( w_i + \widetilde{p}(t) - p(t) ) \cap \B_{a_t} (\bar{u}) $ for $ i = 1, 2 $.
By using \eqref{perturbation-relation2}, one obtains $ u_i \in  \widetilde{G_t} ^{-1} (w_i) $. In fact, $ u_i =  \widetilde{G_t} ^{-1} (w_i) \cap \B_{a_t} (\bar{u}) $.
Thus, steps 1. and 2. are proved. \\
But pointwise strong metric regularity of $G_t$ provides more information, that is
$$ \norm{u_1 - u_2} \leq \kappa_t \norm{w_1 - w_2}. $$
Therefore, step 3. is also proved.
\end{proof}

\begin{rem} \label{Perturbation Effect on the Auxiliary Map-remark}
\textbf{(a)} A careful look at the proof reveals that the lemma could be also expressed in the following way:
\begin{center}
If $G_t$ is SMR at $\bar{u}$ for $0$, then $\widetilde{G_t}$ is SMR at $\bar{u}$ for $  p(t) - \widetilde{p}(t) $.
\end{center}
In this case, $\epsilon$ could be as big as $b_t$.\\
In fact, in this case one can consider \cite[Theorem 3G.3, p. 194]{implicit} with $ F = G_t,~ \rfp{x}{y} =(\bar{u}, 0 ), ~ \kappa = \kappa_t $, and $g(\cdot) = p(t) - \widetilde{p}(t) $ which is a constant function with respect to $u$, so is obviously Lipschitz with any $ \mu < \kappa^{-1} $, and immediately obtain the SMR at $\bar{u}$ for
$  p(t) - \widetilde{p}(t) $ of the map $ ( G_t + p(t) - \widetilde{p}(t) ) $ which is exactly $\widetilde{G_t}$.\\
\textbf{(b)} Under the assumptions of Theorem \ref{claim2}, we would have uniform strong metric regularity for $G_t$ at $z(t) $ for $0$ and the proof shows that we obtain uniform strong metric regularity for
 $\widetilde{G_t}$ at $z(t) $ for $ p(t) - \widetilde{p}(t) $, too.
\end{rem}

Finally, we have provided enough instruments to declare the main result of this section, that is the existence of a solution trajectory $\widetilde{z} (\cdot) $ close to $z(\cdot) $ that inherits its continuity properties.
We may recall that, under the assumptions of Theorem \ref{claim2}, existence of positive constants $a, b $, and $\kappa$ not depending on $t$ is guaranteed for uniform strong metric regularity. Since the following theorem satisfies those assumptions, we will use the uniform constants without ambiguity.

\begin{thm} [\textbf{Existence of a Continuous Trajectory for the Perturbed Problem}] \label{Existence of a Continuous Trajectory for the Perturbed Problem}
For the generalized equations \eqref{pge}, and \eqref{perturbed GE} and the corresponding solution mappings \eqref{solution mapping for time-varying case}, and \eqref{perturbed S},  assume that
\begin{enumerate}
\item [(i)] $z(\cdot) $ is a given continuous solution trajectory (for $S$);
\item [(ii)] $p (\cdot)$ and $ \widetilde{p}(\cdot) $ are continuous functions such that for every $t \in [ 0, \, 1 ] $, $ \norm { \widetilde{p} (t) - p(t) } < \epsilon $ (with
$ \epsilon < b/4$);
\item [(iii)] $G_t$ is pointwise strongly metrically regular at $z(t) $ for $0$.
\end{enumerate}
Then there exists a continuous solution trajectory $\widetilde{z} (\cdot) $ for $\widetilde{S} $ such that, for every $t \in [ 0, \, 1 ] $, we have $ \norm { \widetilde{z} (t) - z(t) } <  \frac{4 a \epsilon}{b} $.
\end{thm}

\begin{proof}
We will present two proofs for this theorem, both are constructional methods, yet with different approaches. Remark \ref{another proof} will provide a comparison between the methods.

\emph{Method 1. Pointwise construction:}\\
Consider an arbitrary $t_0 \in [0, \, 1]$. Since $ \big( t_0, z(t_0) \big) \in \gph{\, S}$ and $G_{t_0}$ is strongly metrically regular at $z(t_0)$ for $0$, by using Theorem \ref{claim2}, we obtain that the mapping
\begin{equation*}
\B_{b} (0) \ni y \longmapsto G_{t_0} ^{-1} (y) \cap \B_{a} (z(t_0) )
\end{equation*}
is single-valued and Lipschitz continuous with constant $\kappa$.
Let $y_0 =\widetilde{p}(t_0) - p(t_0) $. For $\epsilon$ small enough (i.e. $\epsilon < b / 4$), we have $y_0 \in \B_{b} (0)$. Let
\begin{equation} \label{Existence of a Continuous Trajectory for the Perturbed Problem-01}
\widetilde{z} (t_0) := G_{t_0} ^{-1} (y_0) \cap \B_{a} (z(t_0) ).
\end{equation}
Note that the right-hand side of this expression is a singleton and so
$ \widetilde{z} (t_0) $ is exactly determined without ambiguity. Let us check if $ (t_0, \widetilde{z} (t_0) ) \in \gph{\, \widetilde{S} } $ or, equivalently, $ 0 \in \widetilde{G_{t_0}} (\widetilde{z} (t_0))$.\\
From the definition of $ \widetilde{z} (t_0) $ we have $ y_0 \in G_{t_0} ( \widetilde{z} (t_0) ) = f(\widetilde{z} (t_0) ) + F(\widetilde{z} (t_0) ) - p(t_0)$. Then, from
 $y_0 =\widetilde{p}(t_0) - p(t_0) $, one gets  $\widetilde{p}(t_0) \in f(\widetilde{z} (t_0) ) + F(\widetilde{z} (t_0) ) $ or
$ 0 \in \widetilde{G_{t_0}} (\widetilde{z} (t_0)) $.
Since $t_0$ is an arbitrary point in $[0, 1]$,
$$ t \longmapsto G_{t} ^{-1} \big( \widetilde{p}(t) - p(t) \big) \cap \B_{a} (z(t) )$$
defines a single-valued map $\widetilde{z}(\cdot)$.\\
To prove the continuity, consider a sequence $(t_n) \in [0, 1]$ converging to $t_0$. By continuity of $ \widetilde{p}(\cdot)$ and $p(\cdot)$ we know that
$ y_n := \widetilde{p}(t_n) - p(t_n) ~ \longrightarrow ~ y_0 = \widetilde{p}(t_0) - p(t_0) $.\\
By definition, $ \widetilde{z} (t_n) := G_{t_n} ^{-1} (y_n) \cap \B_a (z(t_n) )$. Remark \ref{G-relations} yields that
$$ \widetilde{z} (t_n) \in G_{t_0} ^{-1} ( \, y_n + p(t_n) - p(t_0) \, ).$$
On the other hand, $  \widetilde{z} (t_n) \in \B_a (z(t_n) ) $. We claim that $ \widetilde{z} (t_n) \in \B_a (z(t_0) ) $.\\
Indeed, for $n$ large enough, one can have the following:
\begin{equation} \label{construction-M1}
\norm{p(t_n) - p(t_0)} < \frac{b}{4},~ \norm{z(t_n) - z(t_0)} < \frac{a}{2}.
\end{equation}
Considering Lemma \ref{ratio law}, with $b' = b/2$, and $a' = a/2$, we obtain that the mapping
$$ \B_{b/2} (0) \ni y \longmapsto G_{t_0} ^{-1} (y) \cap \B_{a/2} (z(t_0) ) $$
is single-valued and Lipschitz continuous with Lipschitz constant $\kappa$.
Now observing that
\begin{equation*}
\begin{split}
\norm{ y_n + p(t_n) - p(t_0) } & \leq ~ \norm{ \widetilde{p}(t_n) - p(t_n) } + \norm{ p(t_n) - p(t_0) } \\
 & < \epsilon + \frac{b}{4} \, < \, \frac{b}{2},
\end{split}
\end{equation*}
we can define $ w_n := G_{t_0} ^{-1} \big( y_n + p(t_n) - p(t_0) \big) \cap \B_{a/2} (z(t_0) ) $ without ambiguity.\\
On the one hand, Remark \ref{G-relations} implies that $w_n \in  G_{t_n} ^{-1} ( y_n)$.\\
On the other hand, $w_n \in \B_{a/2} (z(t_0) ) \subset \B_{a} (z(t_n) ) $. Thus, $w_n \in G_{t_n} ^{-1} ( y_n) \cap \B_{a} (z(t_n) )$.\\
Since $G_{t_n} ^{-1}$ is single-valued and Lipschitz continuous when restricted to $ \B_{b} ( 0 ) \times \B_{a} (z(t_n) ) $, we obtain the equality
$ w_n = G_{t_n} ^{-1} ( y_n) \cap \B_{a} (z(t_n) ) $, and thus, $w_n = \widetilde{z} (t_n) $.
Therefore, $ \widetilde{z} (t_n) \in \B_{a/2} (z(t_0) ) \subset \B_{a} (z(t_0) ). $\\
The strong metric regularity of $G_{t_0} $ implies that
\begin{equation*}
\norm{ \widetilde{z} (t_n) - \widetilde{z} (t_0)} \, \leq \, \kappa \norm{ \, \paren{y_n + p(t_n) - p(t_0)} - y_0 \, }.
\end{equation*}
Hence, $ \norm{ \widetilde{z} (t_n) - \widetilde{z} (t_0)} $ converges to zero as $ n \rightarrow \infty $. \\
It only remains to remind that the estimate for the difference $\norm { \widetilde{z} (t) - z(t) }$ is a straightforward consequence of the way we constructed $\widetilde{z}$.
Indeed, let $ \displaystyle r( \epsilon ) := \sup_{ t \, \in \, [0, 1] } \norm{ \widetilde{z} (t) - z(t) } $. Starting from $ \epsilon < \frac{b}{4} $, we obtained $ r(\epsilon) < a $. If we let $ \epsilon < \frac{b}{8} $, a deeper look into the proof reveals that we can obtain $ r( \epsilon ) < \frac{a}{2} $, and so on. Thus, the distance $ \norm{ \widetilde{z} (t) - z(t) } $ (for every $t \, \in \, [0, 1]$) is controlled linearly by $ \epsilon $ and the proof is complete.

\emph{Method 2. Construction over an interval:}\\
Fix $t \in [0, 1]$, and let $b$ smaller if necessary such that $ \kappa b < a $. This will not affect the uniform strong metric regularity of $ G_t $ guaranteed by the assumptions of this theorem and Theorem \ref{claim2}. The uniform continuity of $ \widetilde{p} (\cdot)$, and $ z(\cdot) $ allows us to choose $ \rho >0 $ sufficiently small and independent of $t$, such that for any $ \tau \in (t - \rho, \, t + \rho ) $, the following hold:
\begin{equation} \label{construction-01}
\norm{ \widetilde{p}(\tau) - \widetilde{p} (t) } < \dfrac{b}{4}, ~~ \norm{ z(\tau) - z(t) } < \dfrac{a}{2}.
\end{equation}
Then, for any $ \tau \in (t - \rho, \, t + \rho ) $ the continuity of $ \widetilde{p}(\cdot) $ and its closeness to $p(\cdot) $ implies that
\begin{equation*}
\norm{\widetilde{p}(\tau) - p(t)} \, \leq \, \norm{\widetilde{p}(\tau) -\widetilde{p}(t) } + \norm{\widetilde{p}(t) - p(t)} < \dfrac{b}{2},
\end{equation*}
and therefore, by using Lemma \ref{ratio law} with $ b' = \frac{b}{2}$, and $ a' = \frac{a}{2}$, we obtain that the set $  G_t ^{-1} \big(\, \widetilde{p}(\tau) - p(t) \,\big)  \cap \B_{\frac{a}{2}} (z(t) ) $ is a singleton. Thus, we can define
\begin{equation} \label{Existence of a Continuous Trajectory for the Perturbed Problem-02}
\widetilde{z} (\tau) := G_t ^{-1} \big(\, \widetilde{p}(\tau) - p(t) \,\big)  \cap \B_{\frac{a}{2}} ( z(t) ) \mathrm{~~~for~ any~~ } \tau \in (t - \rho, \, t + \rho ),
\end{equation}
without ambiguity. In order to prove the continuity of this function, consider a sequence $(\tau_n)$ in $(t - \rho, \, t + \rho )$ such that $ \tau_n \longrightarrow \tau $.
Then, from the Lipschitz continuity of $ G_t ^{-1} ( \cdot ) \cap \B_a ( z(t) ) $ over $ \B_b ( 0 ) $ we obtain that
\begin{equation*}
\begin{split}
\norm{ \widetilde{z}(\tau_n) - \widetilde{z}(\tau) } & =
 \norm{ \Big[ G_t ^{-1} \big(\, \widetilde{p}(\tau_n) - p(t) \,\big)  \cap \B_a ( z(t) )  \Big] - \Big[  G_t ^{-1} \big(\, \widetilde{p}(\tau) - p(t) \,\big)  \cap \B_a ( z(t) ) \Big] } \\
& \leq \, \kappa \norm{ ~ \widetilde{p}(\tau_n) - p(t) - \paren{\widetilde{p}(\tau) - p(t) }\, } \\
& \leq \, \kappa \norm{ ~ \widetilde{p}(\tau_n) - \widetilde{p}(\tau) \, }.
\end{split}	
\end{equation*}
The continuity of $ \widetilde{p} $ implies that $ \norm{ \widetilde{z}(\tau_n) - \widetilde{z}(\tau) }  \longrightarrow 0 $ as $ \tau_n \longrightarrow \tau $. \\
It remains to show that $ \widetilde{z} $ is (part of) a solution trajectory, that is, $ \big( \tau, \widetilde{z}(\tau) \big) \in \gph{\, \widetilde{S}} $.\\
Since $ \widetilde{z} (\tau) \in G_t ^{-1} \big(\, \widetilde{p}(\tau) - p(t) \, \big) $, from Remark \ref{G-relations} we get
$ \widetilde{z} (\tau) \in G_{\tau} ^{-1} \big(\, \widetilde{p}(\tau) - p(\tau) \, \big) $. Then, Equation \eqref{perturbation-relation2} implies that
$ \widetilde{z} (\tau) \in \widetilde{G_{\tau}} ^{-1} (0)$  for any $ \tau \in (t - \rho, \, t + \rho ) $, or equivalently, $ \widetilde{z}(\tau) \in \widetilde{S}(\tau) $.\\
Up to now, we have proved that for each $ t \in [0, 1]$, we can find a solution trajectory in the interval $ (t - \rho, \, t + \rho ) $. It remains to show that this construction over different intervals remains consistent.
To be more clear, let us consider two points $t_1$, and $t_2$, with corresponding trajectory pieces $\widetilde{z_1}$, and
$\widetilde{z_2}$. Suppose that $t_1 < t_2 $ and let us consider the situation where $ t_2 - \rho < \tau < t_1 + \rho $. We should prove that
$ \widetilde{z_1} (\tau) = \widetilde{z_2} (\tau) $.

By definition, $ \widetilde{z_i} (\tau) =  G_{t_i} ^{-1} \paren{\, \widetilde{p}(\tau) - p(t_i) \,} \cap \B_a (z(t_i) ) $ for $i = 1, 2 $, and as already shown, Remark \ref{G-relations}, and Equality \eqref{perturbation-relation2} imply that $ \widetilde{z_i} (\tau) \in \widetilde{G_{\tau}} ^{-1} (0) $ for $i = 1, 2 $.\\
On the other hand, the continuity of $ z (\cdot) $, and inequalities in \eqref{construction-01} reveal that
\begin{equation*}
\norm{ z(\tau) - \widetilde{z_i}(\tau) } \, \leq \,  \norm{ z(\tau) - z(t_i) } + \norm{ z(t_i) - \widetilde{z_i}(\tau) } \\
< \, \dfrac{a}{2} \, + \, \dfrac{a}{2}.
\end{equation*}
Thus, $ \widetilde{z_i}(\tau) \in \B_a (z(\tau) ) $ for $i = 1, 2 $. Using Remarks \ref{Perturbation Effect on the Auxiliary Map} and \ref{Perturbation Effect on the Auxiliary Map-remark} for $ \widetilde{G_{\tau}}$, we obtain that the mapping
$ \widetilde{G_{\tau}} ^{-1} (\cdot) \cap \B_a (z(\tau) ) $ is single-valued and Lipschitz continuous over $\B_{\frac{3b}{4}} (0)$. So, $ \widetilde{z_1} (\tau) = \widetilde{z_2} (\tau) $, and the proof is complete.
\end{proof}

\rem \label{another proof}
\textbf{(a)} A thorough observation reveals that in fact, the two methods produce the same function mainly because of the single-valuedness of the mapping
\begin{equation*}
\B_b (0) \ni y \longmapsto G_t ^{-1} (y) \cap \B_a (z(t) ).
\end{equation*}
To be more precise, let us denote the trajectory obtained from \emph{Method 1.} by $\widetilde{z}_{M1}$, and the other one by $ \widetilde{z}_{M2} $.
Consider an arbitrary point $ t \in [0, 1] $ and a neighborhood $ (t - \rho, \, t + \rho ) $ with $ \rho > 0 $ defined in such a way that \eqref{construction-01} holds.
First observe that from Equations \eqref{Existence of a Continuous Trajectory for the Perturbed Problem-01}, and \eqref{Existence of a Continuous Trajectory for the Perturbed Problem-02} we obtain immediately that $\widetilde{z}_{M1} (t) = \widetilde{z}_{M2} (t)$. \\
Now for any $ \tau \in (t - \rho, \, t + \rho )$, we have
$ \widetilde{z}_{M1} (\tau) = G_{\tau} ^{-1} \big( \widetilde{p}(\tau) - p(\tau) \big) \cap \B_a (z(\tau) ) $.\\
We have already seen in proof \emph{Method 1.} of the previous theorem that when inequalities in \eqref{construction-M1} are satisfied (which is the case, by Condition \ref{construction-01} on
$ \rho$), it is possible to conclude that $ \widetilde{z}_{M1} (\tau) \in B_a (z(t) )$. \\
On the other hand, Remark \ref{G-relations} implies that $ \widetilde{z}_{M1} (\tau) \in G_t ^{-1} (\widetilde{p}(\tau) - p(t))$. Thus,
$ \widetilde{z}_{M1} (\tau) \in G_t ^{-1} (\widetilde{p}(\tau) - p(t))\cap B_a (z(t)) $, and by strong metric regularity of $G_t$ we can obtain the desired equality
$ \widetilde{z}_{M1} (\tau) = \widetilde{z}_{M2} (\tau) $.\\

\textbf{(b)} It is worth mentioning that the method of construction over intervals shows explicitly that Lipschitz continuity of $ \widetilde{z} (\cdot) $ could be easily obtained from Lipschitz continuity of $ \widetilde{p}(\cdot) $. But this is not something new or more than what we can obtain from the method of pointwise construction, as it was implicitly mentioned there, too. Indeed, in view of Lemma \ref{Perturbation Effect on the Auxiliary Map}, Proposition \ref{claim1}, and Corollary \ref{claim1-corollary}, we get the same result.

\begin{eg}
Let us consider the simple circuit in Figure \ref{fig: A circuit with DIAC} with a \emph{DIode for Alternative Current} (DIAC) whose $ i-v $ characteristic is given, a DC-bias $V_s = 28 \, v $, an AC signal source $v_s (t) = 2.5 \sin (4 \pi t) $, and a resistor $ R = 220 \, \Omega $. \\
\begin{figure}[ht]
	\centering
		\includegraphics[width=8cm]{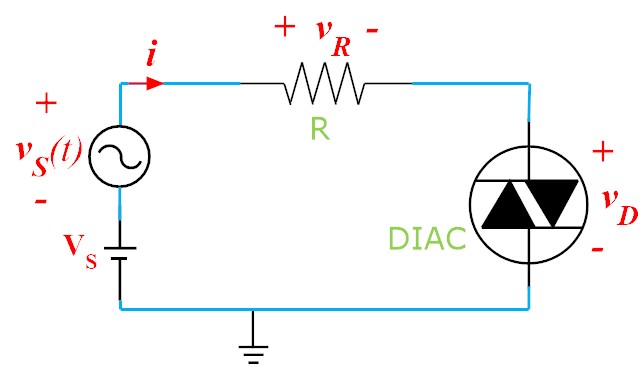}
		$ ~~~~ $
		\includegraphics[width=5.5cm]{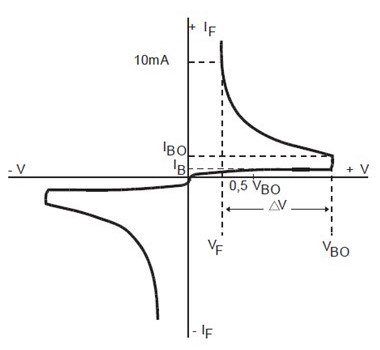}
	\caption{A circuit with DIAC, and its $ i-v $ characteristic}
	\label{fig: A circuit with DIAC}
\end{figure}
Since the line in the first part of the DIAC characteristic is very steep ($I_{B} < 1 \, \mu A, v_D = 16 \,v $), and $ I_{BO} = 100 \, \mu A $, we use a simplified model for the $ i-v $ characteristic, knowing that it does not interfere with our calculations:
{\small
\[
F_{ DIAC } (z) : = \left\{
\begin{array}{lcr}
- \dfrac{a(-z - 10^{-4}) - b}{c(-z - 10^{-4}) - d}			& &  z < -10^{-4} \\
-32 							& & -10^{-4} \leq z < 0 \\
\big[-32, \, 32 \big ] 	& & z = 0 \\
32 							& & 0 < z \leq 10^{-4} \\
 \dfrac{a(z - 10^{-4}) - b}{c(z - 10^{-4}) - d}			& &  z > 10^{-4}
\end{array} \right.
\]}where $a = 15 \, c$, $ b = 32 \, d$, and $ c = -252.52 \, d$. We assumed $ d = 0.1 $ in simulations.
Using KVL, KCL, and characteristics of components one obtains a generalized equation of the form \eqref{pge} with $p(t) = V_s + v_s (t) $, $ f(z) = Rz $  and $F$ is the $F_{ DIAC }$ map. In order to ease the calculations we use a simplification technique to rearrange the single-valued and set-valued terms in the following form
{\small
\begin{equation*}
f^*(z) : = \left\{
\begin{array}{lr}
R z - \dfrac{a(-z - 10^{-4}) - b}{c(-z - 10^{-4}) - d} + 32			 &  z < -10^{-4} \\
R z 																							 & -10^{-4} \leq z \leq 10^{-4} \\
R z + \dfrac{a(z - 10^{-4}) - b}{c(z - 10^{-4}) - d}	-32				 &  z > 10^{-4}
\end{array}, \right. ~
F^*(z) : = \left\{
\begin{array}{lr}
- 32			&   z < 0 \\
\big[ -32, \, 32 \big]			 & z = 0 \\
32				&   z > 0
\end{array} \right.
\end{equation*}}
Looking into Figure \ref{fig: DIAC-G_t} (left) and considering the fact that $ 25.5 \leq \norm{p(t)} \leq 30.5 $, it is clear that three isolated solution trajectories could be specified in the areas of non-activated ($ z = 0 $), negative resistance ($ 5.6 \times 10 ^{-4} \leq z \leq 36 \times 10^{-4} $), and forward conducting ($ z > 16 \times 10^{-3} $); see Figure \ref{fig: Solution trajectories}. For the rest of this example, we will focus on $z_2 (\cdot)$.

\begin{figure}[ht]
		\includegraphics[width=11cm]{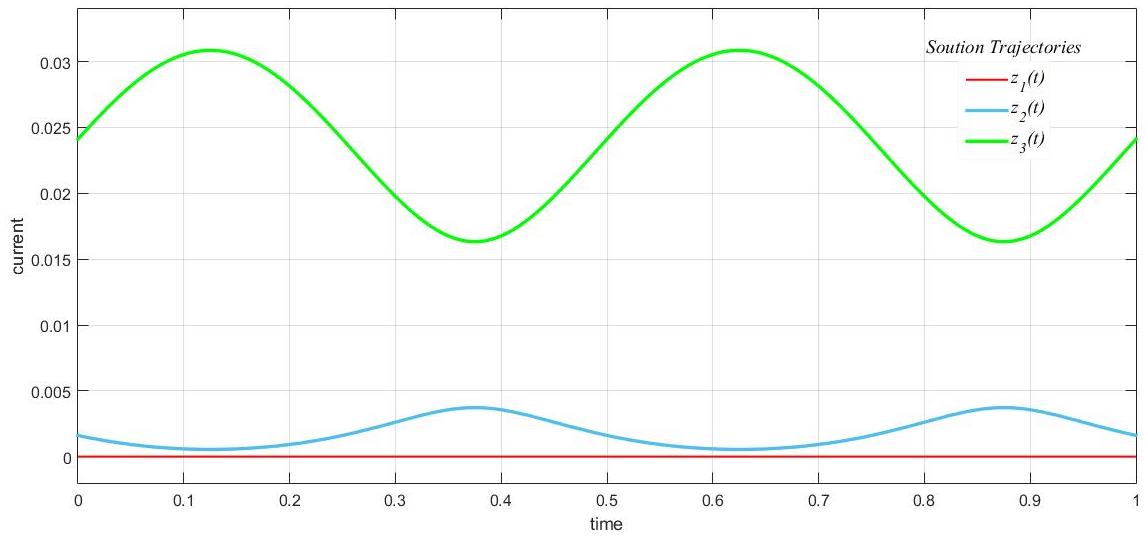}
		\caption{Solution trajectories of the DIAC circuit}
	\label{fig: Solution trajectories}
\end{figure}

\begin{figure}[ht]
	\centering
		\includegraphics[width=7cm]{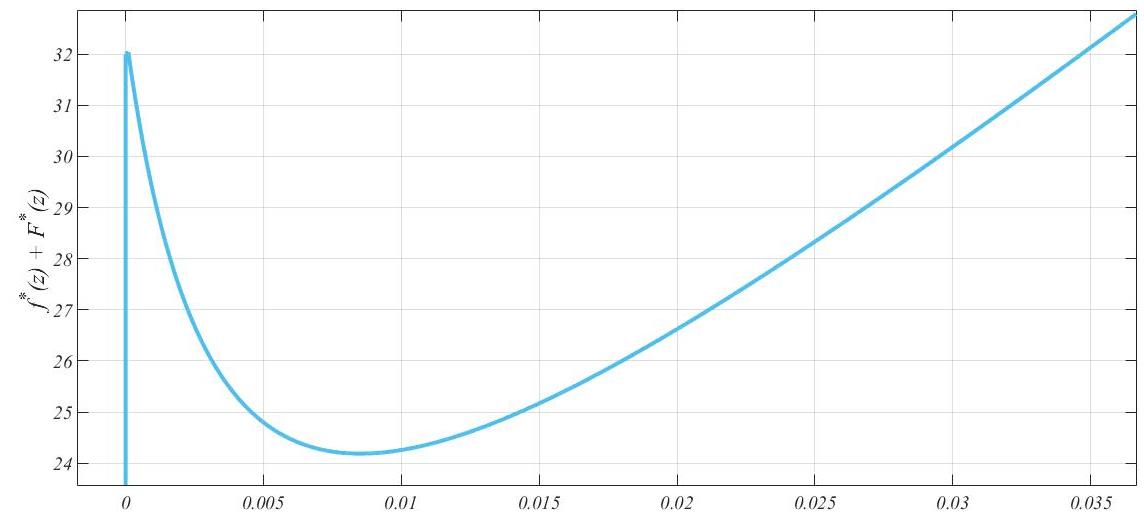}
		$~~~$
		\includegraphics[width=7cm]{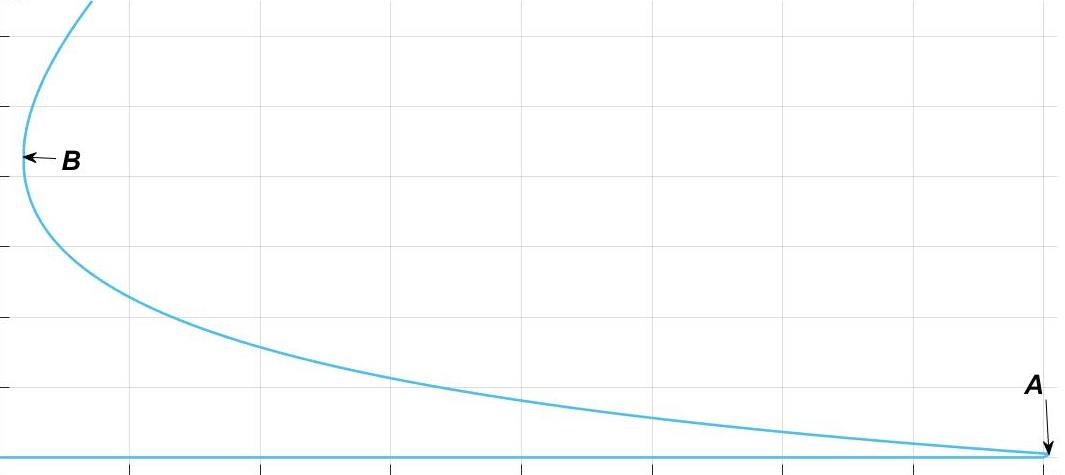}
		\caption{ $ f + F $ (left) and $G_t ^{-1}$ inverse (right) for the DIAC circuit}
	\label{fig: DIAC-G_t}
\end{figure}

In order to check the pointwise strong metric regularity of $ G_t $ at various points $\bar{z}$ for $0$, regarding \citep[Corollary 3.5.7]{Iman2017}, since all the other assumptions are satisfied it only suffices to mention that $f^*$ is not continuously differentiable at $ \bar{z} = \pm 10^{-4} $ and  $ {f^*}^{'}(\bar{z}) $ is positive for any $ \bar{z} > 5.11 \times 10^{-4} $. Hence Proposition \ref{claim1} guarantees the continuity of the solution trajectories $z_2 (\cdot)$, and $z_3 (\cdot)$.\\
To obtain the pointwise constants $a_t$, and $b_t$ that fulfil the SMR definition for $G_t$, by considering the form of $G_t^{-1}$ (Figure \ref{fig: DIAC-G_t} - right) one only needs to avoid reaching points $A$, and $B$ in order to keep away from emptiness or multivaluedness of the localized map. The Lipschitz continuity comes in hand afterwards automatically.
To be more clear, for a point on the solution trajectory $z_2(\cdot)$, say $(0, z_2(\bar{t})) \in \gph{G_t ^{-1}}$, the constants could be obtained as follows:
\begin{eqnarray*}
  a_{\bar{t}}&< &\min \big\{ \norm{z_2(\bar{t}) - y_A} , \norm{z_2(\bar{t}) - y_B} \big\} =  \min \big\{ \norm{z_2(\bar{t}) - 10^{-4}} , \norm{z_2(\bar{t}) - 85 \times 10^{-4}} \big\}, \\
  b_{\bar{t}}&< &\min \big\{ \norm{x_A} , \norm{x_B} \big\} = \min \big\{ \norm{32.022 - p(\bar{t})} , \norm{24.187 - p(\bar{t})} \big\}, \\
  \kappa_{\bar{t}}&>& \dfrac{1}{\norm{ {f^*}^{'} \big( z_2(\bar{t}) \big) }} = \Big( 220 + \dfrac{-43.097}{(25.25 \, z_2(\bar{t}) - 0.097 )^2}\Big)^{-1}.
\end{eqnarray*}
The uniform constants whose existence is guaranteed by Theorem \ref{claim2}, could be computed by considering subintervals of $ [0, 1] $, doing the calculations for each subinterval separately (for more details, see proof of \citep[Theorem 6G.1]{implicit} or \citep[Theorem 4.1.2]{Iman2017}) and finally using the following relations:
\begin{equation*}
\begin{split}
& \kappa := \max \{ \kappa_{t_i} ~|~  i = 1, ..., m \} = 1.66 \times 10^{-4}, \\
& a := \min \big \{ a_{t_i}~|~ i = 1, ..., m \big \} =4.5 \times 10^{-4}, \mathrm{~and~} \\
& b := \min \Big \{~ \dfrac{a}{\kappa} ,~ \min \big \{ b_{t_i} ~|~  i = 1, ..., m \big \} \Big \} = \min \{1.355, \, 1.313 \} = 1.313.
\end{split}
\end{equation*}
For the perturbation problem, let us consider the input signal $\tilde{p}(t) = 27.83 + 2.4 \,  \sin (4 \pi t + \frac{\pi}{64}) $ which includes a DC voltage drop, an amplitude perturbation, and phase shift (or time delay) of the AC signal with respect to the original input signal $p(t)$. \\
Since $ \displaystyle \epsilon = \max_{t} \norm{p(t) - \tilde{p}(t))} = 0.326 < \frac{b}{4} $, in view of Theorem \ref{Existence of a Continuous Trajectory for the Perturbed Problem} one expects to find a continuous solution trajectory $\widetilde{z}_2$ such that $ \norm { \widetilde{z}_2 (t) - z_2 (t) } < \frac{4 a \epsilon}{b} = 4.47 \times 10^{-4} $ for every $t \in [ 0, \, 1 ] $. Numerical calculations confirm that $  \displaystyle \max_{t} \norm{\widetilde{z}_2 (t) - z_2 (t)} = 1.9 \times 10^{-4} $, see Figure \ref{fig: Solution trajectories2}. The simulation and numerical computations has been done with Matlab software.
\begin{figure}[ht]
		\includegraphics[width=14cm]{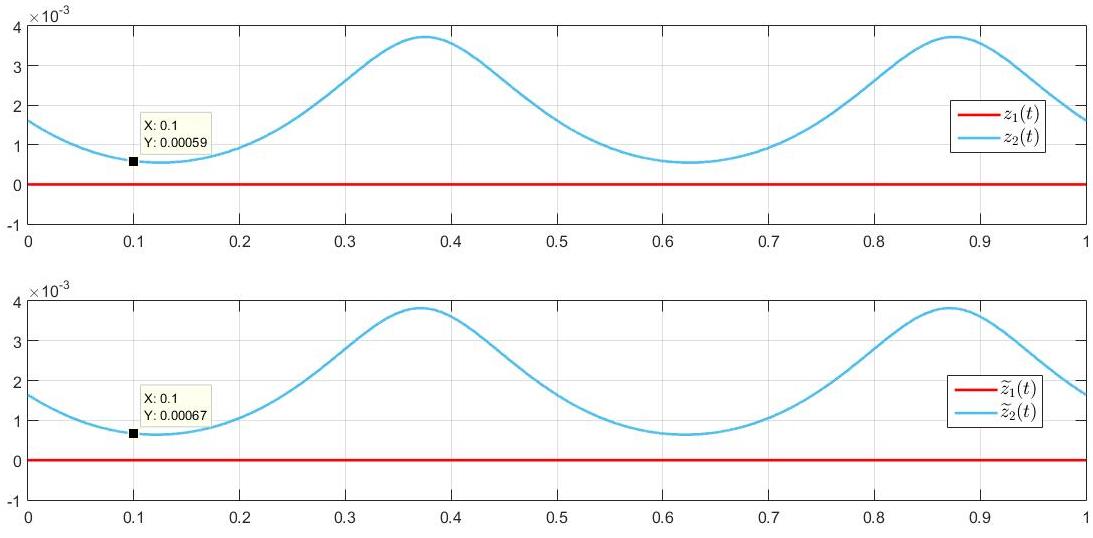}
		\caption{ Solution trajectories $z_1$, $z_2$, $\widetilde{z}_1$, and $ \widetilde{z}_2$ }
	\label{fig: Solution trajectories2}
\end{figure}
\end{eg}
\vspace{0cm}
\subsection*{Acknowledgements}
We wish to thank Professor Radek Cibulka for valuable conversations with the first author during his stay at the Department of Mathematics of the University of West Bohemia.

\label{Bibliography}
\def\bibfont{\small}
\bibliographystyle{amsplain}


\providecommand{\bysame}{\leavevmode\hbox to3em{\hrulefill}\thinspace}
\providecommand{\MR}{\relax\ifhmode\unskip\space\fi MR }
\providecommand{\MRhref}[2]{%
  \href{http://www.ams.org/mathscinet-getitem?mr=#1}{#2}
}
\providecommand{\href}[2]{#2}
\begin{thebibliography}{}

\end{thebibliography}


\begin{thebibliography}{10}

\bibitem{adly2}
S.~Adly and R.~Cibulka, \emph{Quantitative stability of a generalized
  equation}, J. Optim. Theory Appl. \textbf{160} (2014), no.~1, 90--110.

\bibitem{adly}
S.~Adly, R.~Cibulka, and H.~Massias, \emph{Variational analysis and generalized
  equations in electronics}, Set-Valued Var. Anal. \textbf{21} (2013), no.~2,
  333--358.

\bibitem{adly2015newton}
S.~Adly, R.~Cibulka, and H.~V. Ngai, \emph{Newton's method for solving
  inclusions using set-valued approximations}, SIAM J. Optim. \textbf{25}
  (2015), no.~1, 159--184.

\bibitem{adly2016newton}
S.~Adly, H.~V. Ngai, and V.~V. Nguyen, \emph{Newton's method for solving
  generalized equations: {K}antorovich's and {S}male's approaches}, J. Math.
  Anal. Appl. \textbf{439} (2016), no.~1, 396--418.

\bibitem{adly3}
S.~Adly and J.~V. Outrata, \emph{Qualitative stability of a class of
  non-monotone variational inclusions. {A}pplication in electronics}, J. Convex
  Anal. \textbf{20} (2013), no.~1, 43--66.

\bibitem{Artacho20101149}
F.~J. Arag\'on~Artacho and B.~S. Mordukhovich, \emph{Metric regularity and
  {L}ipschitzian stability of parametric variational systems}, Nonlinear Anal.
  \textbf{72} (2010), no.~3-4, 1149--1170.

\bibitem{Bianchi2013279}
M.~Bianchi, G.~Kassay, and R.~Pini, \emph{An inverse map result and some
  applications to sensitivity of generalized equations}, J. Math. Anal. Appl.
  \textbf{399} (2013), no.~1, 279--290.

\bibitem{MRDGE2016}
R.~Cibulka, A.~L. Dontchev, M.~Krastanov, and V.~M. Veliov, \emph{Metrically
  regular differential generalized equations}, Tech. report, Institute of
  Statistics and Mathematical Methods in Economics, Vienna University of
  Technology, 09 2016.

\bibitem{cibulka2016strong}
R.~Cibulka, A.~L. Dontchev, and A.~Y. Kruger, \emph{Strong metric subregularity
  of mappings in variational analysis and optimization}, J. Math. Anal. Appl.
  \textbf{457} (2018), no.~2, 1247--1282.

\bibitem{dontchev2013}
A.~L. Dontchev, M.~I. Krastanov, R.~T. Rockafellar, and V.~M. Veliov, \emph{An
  {E}uler-{N}ewton continuation method for tracking solution trajectories of
  parametric variational inequalities}, SIAM J. Control Optim. \textbf{51}
  (2013), no.~3, 1823--1840.

\bibitem{dontchev2010newton}
A.~L. Dontchev and R.~T. Rockafellar, \emph{Newton's method for generalized
  equations: a sequential implicit function theorem}, Math. Program.
  \textbf{123} (2010), no.~1, Ser. B, 139--159.

\bibitem{implicit}
\bysame, \emph{Implicit functions and solution mappings, {A} view from
  variational analysis}, second ed., Springer Series in Operations Research and
  Financial Engineering, Springer, New York, 2014.

\bibitem{durea2012openness}
M.~Durea and R.~Strugariu, \emph{Openness stability and implicit multifunction
  theorems: applications to variational systems}, Nonlinear Anal. \textbf{75}
  (2012), no.~3, 1246--1259.

\bibitem{ferreira2016unifying}
O.~P. Ferreira and G.~N. Silva, \emph{Unifying the local convergence analysis
  of {N}ewton's method for strongly regular generalized equations}, arXiv
  preprint arXiv:1604.04568 (2016).

\bibitem{ferreira2016kantorovich}
\bysame, \emph{Kantorovich's theorem on {N}ewton's method for solving strongly
  regular generalized equation}, SIAM J. Optim. \textbf{27} (2017), no.~2,
  910--926.

\bibitem{izmailov2014}
A.~F. Izmailov, \emph{Strongly regular nonsmooth generalized equations}, Math.
  Program. \textbf{147} (2014), no.~1-2, Ser. A, 581--590.

\bibitem{Iman2017}
I.~Mehrabinezhad, \emph{Metrically regular generalized equations: A case study
  in electronic circuits}, Ph.D. thesis, University of Milano-{B}icocca, Italy,
  2017.

\bibitem{Mordukhovich}
B.~S. Mordukhovich, \emph{Variational analysis and generalized differentiation.
  {I}. {B}asic theory}, Grundlehren der Mathematischen Wissenschaften
  [Fundamental Principles of Mathematical Sciences], vol. 330, Springer-Verlag,
  Berlin, 2006.

\bibitem{robinson1979generalized}
S.~M. Robinson, \emph{Generalized equations and their solutions. {I}. {B}asic
  theory}, Math. Programming Stud. (1979), no.~10, 128--141, Point-to-set maps
  and mathematical programming.

\bibitem{robinson}
\bysame, \emph{Strongly regular generalized equations}, Math. Oper. Res.
  \textbf{5} (1980), no.~1, 43--62.

\bibitem{robinson1982generalized}
\bysame, \emph{Generalized equations and their solutions. {II}. {A}pplications
  to nonlinear programming}, Math. Programming Stud. (1982), no.~19, 200--221,
  Optimality and stability in mathematical programming.

\bibitem{robinson1983generalized}
\bysame, \emph{Generalized equations}, Mathematical programming: the state of
  the art ({B}onn, 1982), Springer, Berlin, 1983, pp.~346--367.

\bibitem{Sedra}
A.~S. Sedra and K.~C. Smith, \emph{Microelectronic circuits}, Oxford University
  Press, Inc., 5th edition, 2004.

\bibitem{uderzo2009some}
A.~Uderzo, \emph{On some regularity properties in variational analysis},
  Set-Valued Var. Anal. \textbf{17} (2009), no.~4, 409--430.

\end{thebibliography}


\end{document}